\documentclass[letterpaper,11pt,reqno]{amsart}

\usepackage{amssymb}
\usepackage{amsmath} 
\usepackage{latexsym}
\usepackage{amsbsy}
\usepackage{amsfonts}
\usepackage{graphicx}
\usepackage{enumerate}
\usepackage{enumitem}
\usepackage{mathtools}
\usepackage{bbm} 
\usepackage{color}
\usepackage{url}
\usepackage{comment}
\usepackage{appendix}
\usepackage{caption}
\usepackage{subcaption}

\usepackage{tikz}

\def\marginpar#1{\ignorespaces}

\usepackage[square,numbers]{natbib}
\usepackage{hyperref}

\newtheorem{theorem}{Theorem}[section]
\newtheorem{lemma}[theorem]{Lemma}
\newtheorem{corollary}[theorem]{Corollary}
\newtheorem{definition}[theorem]{Definition}
\newtheorem{proposition}[theorem]{Proposition}
\newtheorem{remark}[theorem]{Remark}
\newtheorem{claim}[theorem]{Claim}
\newtheorem{assumption}[theorem]{Assumption}

\numberwithin{equation}{section}

\author[Winston Yu]{{Winston} Yu}
\address{Department of Applied Physics and Applied Mathematics, Columbia University. 
} \email{wy2436@columbia.edu}

\author[Qiang Du]{{Qiang} Du}
\address{Department of Applied Physics and Applied Mathematics, Columbia University. 
} \email{qd2125@columbia.edu}

\author[Wenpin Tang]{{Wenpin} Tang}
\address{Department of Industrial Engineering and Operations Research, Columbia University. 
} \email{wt2319@columbia.edu}

\date{\today}

\begin{document}

\title[Vanishing viscosity approximations]{Beyond separability: convergence rate of vanishing viscosity approximations to mean field games via FBSDE stability}

\begin{abstract}
   We study the vanishing viscosity approximation to mean field games (MFGs) in $\mathbb{R}^d$ with a nonlocal and possibly non-separable Hamiltonian. We prove that the value function converges at a rate of $\mathcal{O}(\beta)$, where $\beta^2$ is the diffusivity constant, which matches the classical convergence rate of vanishing viscosity for Hamilton-Jacobi (HJ) equations. The same rate is also obtained for the approximation of the distribution of players as well as for the gradient of the value function. The proof is a combination of probabilistic and analytical arguments by first analyzing the forward-backward stochastic differential equation associated with the MFG, and then applying a general stability result for HJB equations. Applications of our result to $N$-player games, mean field control, and policy iteration for solving MFGs are also presented.
\end{abstract}

\maketitle    


\section{Introduction}

\quad Mean field games (MFGs) were simultaneously proposed by Lasry and Lions in \cite{LL2006a, LL2006b, LL2007}, and by Huang, Malhame, and Caines in \cite{Huang2006},  for the purpose of modeling a game with a large number of players whose decisions are influenced by the distribution of the other players. Due to the large number of players, each player is assumed to have an infinitesimally small influence on all of the other players. Moreover, if we assume that all players act rationally (meaning that they each solve an optimization problem of some cost functional and act as though all other players are also playing rationally), then the system is said to be in a Nash equilibrium. Now suppose that the agents are playing on the state space $\mathbb{R}^d$. Then, one of the most common formulations of MFGs is as a system of coupled partial differential equations (PDEs), of which the first is a Hamilton-Jacobi-Bellman (HJB) equation solved by $u^{\beta}: [0, T] \times \mathbb{R}^d \to \mathbb{R}$, and the second is a Fokker-Planck equation solved by a flow $\rho^{\beta}$ of probability measures on $\mathbb{R}^d$:
\begin{equation}\label{eq:mfg}
\begin{cases}
    -\partial_t u^{\beta} + H(x, -\nabla u^{\beta}, \rho^{\beta}_t) = \frac{\beta^2}{2} \Delta u^{\beta} & \text{ on } [0, T] \times \mathbb{R}^d, \\
    \partial_t \rho^{\beta}_t + \operatorname{div}_x\{\rho^{\beta}_t \nabla_p H(x, -\nabla u^{\beta}, \rho^{\beta}_t)\} = \frac{\beta^2}{2}\Delta \rho^{\beta}_t & \text{ on } [0, T] \times \mathbb{R}^d, \\
    u^{\beta}(T, x) = g(x, \rho^{\beta}_T), \quad \rho^{\beta}_0(x) = m_0(x) & \text{ on } \mathbb{R}^d,
\end{cases}
\end{equation}
where $\beta \ge 0$ is the idiosyncratic noise intensity \footnote{The coefficient $\beta^2/2$ is also called the diffusivity constant.}, $H$ is a possibly non-separable Hamiltonian, $g$ is the terminal cost, and $m_0$ is the initial distribution. When $\beta > 0$, the system is of second order. The system obtained by sending $\beta \to 0$ is called the {\em vanishing viscosity limit}, which is of first order. 
Refer to Section \ref{sc11} for a literature review on the well-posedness of MFGs \eqref{eq:mfg}.

\quad Second-order MFGs are widely used to model complex systems in economics \cite{household_wealth_PDE_original, carmona, Moll_Slideshow} and engineering \cite{djehiche, huang2007large}. Recently, there has been a surge of interest in first-order MFG models, where the state evolves according to deterministic dynamics. Examples include the traffic flow of pedestrian crowds and autonomous vehicles \cite{GM25, Huang2017}, and Proof-of-Stake cryptocurrency mining \cite[Section 5]{trading_and_wealth_evolution}. As prior works have noted, such as in \cite{Huang2017}, traditional methods (e.g., Newton's method) for solving first-order MFGs tend to be numerically unstable. Various approaches \cite{CS14, CSZ24, Liu21} have recently been proposed to address such problems in solving first-order MFGs with a separable Hamiltonian (see \eqref{eq:separability} below), but with no quantitative guarantees. An obvious approach, as already suggested in \cite{AL20}, is to add a small second-order perturbation that corresponds to the addition of a small idiosyncratic noise, and then to solve the resulting second-order MFG \footnote{This approach is also reminiscent of the Lax-Friedrichs approximation scheme to first-order equations, where numerical viscosity is $\beta = \sqrt{\frac{2 (\Delta x)^2}{\Delta t}}$ with $(\Delta x, \Delta t)$ as the space-time discretization (see e.g., \cite{CCGS21, CCS19, Tadmor84}).}. However, the price of gaining numerical stability from the perturbation is to introduce a source of error depending on the noise intensity $\beta$. Denoting $u = u^0$ for the value function of the first-order MFG, one would expect from the classical theory of Hamilton-Jacobi equations (see e.g., \cite[Section IV]{Crandall1983}) that as $\beta \to 0$, $u^{\beta} \rightarrow u$ in some topology. But the classical theory of viscosity solutions cannot immediately provide a convergence rate with respect to $\beta$ for MFGs on account of the coupling with the Fokker-Planck equation.

\quad The purpose of this paper is to provide a quantitative convergence rate to vanishing viscosity limit of the MFG \eqref{eq:mfg} with a general, possibly non-separable Hamiltonian. Previous work \cite{Tang2023b} has studied the convergence rate of vanishing viscosity limits of MFGs with a separable Hamiltonian on the torus $\mathbb{T}^d$; see the end of Sections \ref{sc11} and Remark \ref{rk:comparison} for discussions. Now let us briefly describe our results: we prove that under suitable conditions on the model parameters, for any compact set $\mathcal{K} \subseteq \mathbb{R}^d$, 

$$\|u^{\beta} - u\|_{L^{\infty}([0, T] \times \mathcal{K})} \leq C_{\mathcal{K}} \beta,$$ 

for some constant $C_{\mathcal{K}}$. Under certain regularity conditions, we can show that $C_\mathcal{K}$ grows at most quadratically in the diameter of $\mathcal{K}$. In other words, $u^{\beta} \rightarrow u$ at a rate of $\mathcal{O}(\beta)$ in the topology of uniform convergence on compact sets (Theorem \ref{main_result}), which matches the convergence rate from the classical viscosity theory of Hamilton-Jacobi equations. As intermediary steps, we prove: (1) when the initial condition is bounded, $\nabla u^\beta$ converges to $\nabla u$ at a rate of $\mathcal{O}(\beta)$ in the $L^{\infty}([0, T] \times U)$ metric, for any large enough bounded set $U \subseteq \mathbb{R}^d$ (Theorem \ref{convergence_of_gradient}), as well as in the $L^2(\rho_t)$ metric, uniformly in $t$ (Corollary \ref{cor:L2_convergence_of_grad_ubeta}), and (2) $\rho_t^\beta$ converges to $\rho_t$ in the $2$-Wasserstein distance (Corollary \ref{convergence_of_rho}). Our analysis specifically requires neither the Lasry-Lions nor the displacement monotonicity condition; see Remark \ref{rk:comparison} for a discussion. In Section \ref{sc5}, we also show how to apply our result to various problems, such as particle system approximations and policy iteration for solving MFGs.

\quad Here we give a quick outline of the proof, which is a combination of probabilistic and analytical arguments. From the classical stability theory of PDEs, one might suspect that the difference between $u^{\beta}$ and $u$ is controlled by the difference in the coefficients, i.e., $H(\cdot, \cdot, \rho^{\beta}_t)- H(\cdot, \cdot, \rho_t)$, $g(\cdot, \rho^{\beta}_T)- g(\cdot, \rho_T)$, and $\beta$. However, the dependence on the measure in the first four terms complicates the analysis, because $\rho^{\beta}$ and $\rho$ satisfy their own PDEs that depend on $u^{\beta}$ and $u$, respectively. We avoid this issue by instead analyzing the forward-backward stochastic differential equation (FBSDE) system (defined in \eqref{eq:fbsde}) associated with the MFG system \eqref{eq:mfg}. Our analysis consists of three steps:
\begin{enumerate}[itemsep = 3 pt]
    \item Using the FBSDE representation of the MFG, and temporarily assuming that the initial condition is bounded, we control the $L^2$ difference between $\rho^{\beta}_t$ and $\rho_t$ in terms of $\beta$ and $\|\nabla u^{\beta} - \nabla u\|_{L^{\infty}(U)}$ (Lemma \ref{convergence_in_L2_of_Xbeta_by_beta_grad_ubeta}) for some large enough set $U \subseteq \mathbb{R}^d$.
    \item Then, using the decoupling field of the FBSDE, we prove that $\|\nabla u^{\beta} - \nabla u\|_{L^{\infty}(U)} = \mathcal{O}(\beta)$ (Theorem \ref{convergence_of_gradient}). To remove the assumption that the initial condition is bounded, we apply a stability result for FBSDEs. This implies the convergence of $\rho^{\beta}_t$ to $\rho_t$ at a rate of $\mathcal{O}(\beta)$ in the $L^2$, and hence $W_2$, metric (Corollary \ref{convergence_of_rho}).
    \item We finally apply the previous two steps, in combination with a general PDE stability result (Theorem \ref{HJ_PDE_stability}), to the PDE formulation of MFGs, in order to derive a convergence rate of $\mathcal{O}(\beta)$ for $u^\beta$ to $u$ in the topology of uniform convergence on compact sets (Theorem \ref{main_result}). 
\end{enumerate}
Finally, we comment that while we mostly use the $L^2$ and $W_2$ distances in the statements of our results, we expect that our results should easily extend to $L^p$ and $W_p$ distances when the initial distribution $m_0$ is only $p$-integrable, for $p \in [1, 2)$.

\medskip
{\bf Organization of the paper}: 
The remainder of this paper is organized as follows. Section \ref{sc11} provides a literature review on MFGs,  and compares the result in this paper with prior work. In Section \ref{sc2}, we formally define the problem and collect some assumptions for our result. Section \ref{sc3} proves the convergence rate of $\rho_t^\beta$ and $\nabla u^\beta$, and Section \ref{sc4} proves the convergence rate of $u^\beta$. In Section \ref{sc5}, we give several applications of our result.  Examples and numerical experiments are presented in Section \ref{sc6}. 
We make some concluding remarks in Section \ref{sc7}.

\subsection{Literature review} \label{sc11}
The well-posedness of MFGs has been studied extensively, particularly for the case of a separable Hamiltonian, in which the momentum and the measure arguments are additively separated:
\begin{equation}\label{eq:separability}
    H(x, p, \mu) = H_0(x, p) - f(x, \mu).
\end{equation}
Here $H$ can either be a local or nonlocal function of the measure argument. In the separable, local case and when $\beta = 0$, \cite{Cardaliaguet2015b, Cardaliaguet2015c, Ferreira2021} are some of the major works proving the well-posedness. In the separable, nonlocal case and when $\beta = 0$, \cite{Cardaliaguet2021} proved well-posedness of the MFG system, provided that the Lasry-Lions monotonicity condition holds. 

\quad The separable case is a strong structural assumption, upon which much of the previous literature relied. However, many applications may go beyond this assumption. In the economic model proposed by \cite{household_wealth_PDE_original, Moll_Slideshow}, despite the model of the agent being relatively simple to formulate, such an agent corresponds to a MFG whose Hamiltonian is not separable. Another example is provided by \cite{trading_and_wealth_evolution}, where the Hamiltonian has a term containing the product of the price process (which is a function of the player distribution) and the action of the player. Finally, \cite{AP2018} is one of the works that models mean field games with congestion, where players are penalized based on the density of other players at the current position; as a result of congestion penalizing movement, the Hamiltonian some function of the momentum divided by another function of the density. In order to make sense of the MFGs arising in these applications, the well-posedness of MFGs with a non-separable Hamiltonian must first be established. A breakthrough was made by \cite{Gangbo2022}, and later \cite{Meszaros2024}, for their proposal of a new condition, called displacement monotonicity, under which well-posedness of the MFG system \eqref{eq:mfg} can be proved for all $\beta \geq 0$. To the best of our knowledge, the work that proves well-posedness under the least restrictive regularity assumptions on $H$ and $g$ is \cite{Bansil2024b}, whose main assumption, other than displacement monotonicity, is the uniform boundedness of the second derivatives of $H$ and $g$.

\quad Some earlier works take a probabilistic approach to MFGs as well. \cite{CL15} considers a probabilistic formulation of the MFG, where the volatility is uncontrolled. Their Remark $7.12$ is similar to our FBSDE system \eqref{eq:fbsde}, though our equation of the adjoint process is for the gradient of the value function, not for the value function itself. The later work of \cite{Lacker15} also studies MFGs from a probabilistic perspective, and they prove the existence, though not its uniqueness, of a MFG solution where the volatility is controlled. See \cite{CD18} for further developments in this direction.

\quad While the classical setting of the convergence rate of vanishing viscosity approximations to pure Hamilton-Jacobi equations has been studied extensively, two recent papers \cite{CG25, CD25} were motivated by applications to mean-field control to provide an even sharper convergence rate of $\mathcal{O}(\beta^2 \log(\beta^2))$. Central to both papers is an estimate of the integral of the Laplacian of the value function, with respect to the solution of an adjoint equation \cite{Lin2001}, over $\mathbb{T}^d$ in \cite{CG25} and $\mathbb{R}^d$ in \cite{CD25}. Although they do not apply their results to mean field games, especially ones with a more general, non-separable Hamiltonian (however, in their setting, "non-separability" would mean that the Hamiltonian's momentum and time arguments cannot be separated like Equation \eqref{eq:separability}), it would be interesting to apply their technique to our setting as well.

\medskip
{\bf Comparison to previous work}:
The only previous work addressing the convergence rate of vanishing viscosity approximations to MFGs is \cite{Tang2023b}. Its main assumption is that of a separable Hamiltonian, which (along with the terminal cost function) satisfies the Lasry-Lions monotonicity condition. In contrast, we do not assume any monotonicity condition until one is needed for the well-posedness of the MFG system \eqref{eq:mfg}. Moreover, our result addresses the case of a non-separable Hamiltonian which is nonlocal in the measure argument. We prove that the convergence rate for $\{u^{\beta}\}_{\beta > 0}$ is $\mathcal{O}(\beta)$ in $L^{\infty}([0, T] \times \mathcal{K})$ for any compact set $\mathcal{K} \subseteq \mathbb{R}^d$, which improves upon their rate $\mathcal{O}(\beta^{1/2})$ in $L^1(\mathbb{T}^d)$. This (partially) solves Problem 4(a) in \cite{Tang2023b} for MFGs with nonlocal and possibly non-separable Hamiltonians.

\section{Notations, Assumptions and Problem Formulation} \label{sc2}

\subsection{Notations}

For a metric space $X$, let $C^k(X)$ be the space of functions mapping $X$ to $\mathbb{R}$, which are $k$-times differentiable and whose $k$-th order derivatives are continuous. $C^k_c(X)$ is the subset of $C^k(X)$ whose functions are compactly supported. If $X = [0, T] \times \mathbb{R}^n$, then for $f: [0, T] \times \mathbb{R}^n \to \mathbb{R}$, $\nabla f(t, x)$ refers to $\nabla_x f(t, x) = [\partial_i f(t, x)]_{i=1}^n$, and $\nabla^2 f(t, x)$ refers to $\nabla^2_{xx} f(t, x) = [\partial_{ij} f(t, x)]_{i, j = 1}^n$. In particular, $\nabla f$ and $\nabla^2 f$ do not include the partial derivatives with respect to $t$. 

\quad For $p \in [1, \infty]$, a generic measure space $(X, \mathcal{B}, \mu)$, and a metric space $(Y, |\cdot|)$ (which will almost always be Euclidean space $\mathbb{R}^n$ in this paper), $L^p(X, \mathcal{B}, \mu; Y)$ is the space of $Y$-valued functions whose $p$-th power is integrable with respect to $\mu$, i.e., all $f: X \to \mathbb{R}$ such that $\|f\|^p_{L^p} = \int_X |f|^p d\mu < \infty$. If $p = \infty$, then $L^{\infty}(X, \mathcal{B}, \mu)$ is the space of functions $f$ such that there exists a constant $C > 0$ satisfying $\mu(|f| > C) = 0$; the infimum of all such constants is denoted by $\|f\|_{L^{\infty}(X)}$. If we omit a $\sigma$-field $\mathcal{B}$, then it should be clear from context whether it is the Borel $\sigma$-field or an element of some filtration generated by a stochastic process, for instance. We might also omit specifying the measure $\mu$ when it is clear whether it is, for example, Lebesgue measure on $\mathbb{R}^d$ or a probability measure $\mathbb{P}$ on some sample space $\Omega$. If we do not specify a metric space $Y$, then it should be taken $\mathbb{R}$. We will also make use of $L^{\infty}$ spaces of functions mapping $\mathbb{R}^n$ to $\mathbb{R}^m$, denoted by $L^{\infty}(\mathbb{R}^n; \mathbb{R}^m),$ consisting of functions $f: \mathbb{R}^n \to \mathbb{R}^m$ such that there exists a $C$ with $|f| \leq C$ almost everywhere with respect to the Lebesgue measure on $\mathbb{R}^n$ ($|\cdot|$ is the Euclidean norm on $\mathbb{R}^n$). If $\mathbb{R}^m$ is replaced by $\mathbb{R}^{m \times m}$, the space of $m$ by $m$ matrices, then $|\cdot|$ is replaced by the operator norm $\|\cdot\|_{\infty}$ on matrices. 

\quad Now let $(X, \mathcal{B}, \mathbb{P})$ be a probability space. For a random variable $\xi$, $\operatorname{Law}(\xi)$ is the law of $\xi$ with respect to $\mathbb{P}$. For $p \in [1, \infty]$, we write $\mathcal{P}_p(X)$ for the space of probability measures $\mu$ with finite $p$-th moment, i.e., $\int_X |x|^p d\mu(x) < \infty$. On $\mathcal{P}_p(X)$, we define the $p$-Wasserstein distance $W_p$:
\begin{equation*}
W_p(\mu, \nu) = \inf \left\{ \int_{X \times X} |x - y|^p d\pi(x, y) \colon \pi \in \mathcal{P}(X \times X)\text{ has marginals } \mu \text{ and } \nu \right\}^{1/p}.
\end{equation*}

\quad We also introduce the concept of the Wasserstein gradient of a function $U: \mathcal{P}_2(\mathbb{R}^n) \to \mathbb{R}$. For a more extensive introduction, please refer to \cite[Chapter 5]{CD18}, \cite[Section 1.4]{Cardaliaguet2021}, or \cite[Section 6]{Car12}. For $\mu \in \mathcal{P}_2(\mathbb{R}^n)$, the Wasserstein gradient of $U$ at $\mu$ is denoted by $\nabla_\mu U(\mu, \cdot): \mathbb{R}^n \mapsto \mathbb{R}^n$, and it is an element of the closure of gradients of $C^{\infty}(\mathbb{R}^n)$ functions, with respect to the $L^2(\mathbb{R}^n, \mu)$ metric. Moreover, it follows from \cite[Definition 1.11]{Cardaliaguet2021} \footnote{The Wasserstein gradient is defined slightly differently in \cite[Definition 1.11]{Cardaliaguet2021}, but Equation \ref{eq:def_w2_gradient} can be recovered by taking the transport plan $\pi$ to be $(\xi, \xi + \eta)$ in their notation.} that $\nabla_\mu U$ satisfies: for all $\mathbb{R}^n$-valued $L^2$ random variables $\xi$ and $\eta$,
\begin{equation}\label{eq:def_w2_gradient}
    U(\operatorname{Law}(\xi + \eta)) - U(\operatorname{Law}(\xi)) = \mathbb{E}[ \langle \nabla_{\mu} U(\operatorname{Law}(\xi), \xi), \eta \rangle ] + o(\mathbb{E}[|\eta|^2]^{1/2}).
\end{equation}
See also \cite{GT19, WZ17} for related discussions.

\quad Finally, for $a, b > 0$, the symbol $a = \mathcal{O}(b)$, or $a \lesssim b$ means that $a/b$ is bounded from above, as some problem parameter tends to $0$ or $\infty$. Similarly, $a \asymp b$ means that $a/b$ is bounded from below and from above, as some problem parameter tends to $0$ or $\infty$. 

\subsection{Assumptions} \label{sc22}
Unless otherwise said, we work on a filtered probability space $(\Omega, \mathcal{F}, \mathbb{F}, \mathbb{P})$, $\mathbb{F} = \{\mathcal{F}_t\}_{t \in [0, T]}$, generated by a standard $d$-dimensional Brownian motion $B$. Let $H: \mathbb{R}^d \times \mathbb{R}^d \times \mathcal{P}_2(\mathbb{R}^d) \to \mathbb{R}$ be the Hamiltonian, and $g: \mathbb{R}^d \times \mathcal{P}_2(\mathbb{R}^d) \to \mathbb{R}$ be the terminal cost function. We need the following Lipschitz and regularity conditions on $H$ and $g$, as well as a convexity assumption on $H$ for the FBSDE representation \eqref{eq:fbsde} and a well-posedness assumption.

\begin{assumption}[Regularity of $H$]
\label{hamiltonian_assumption_2nd_derivatives}
The derivatives $\nabla_\mu H$, $\nabla^2_{xx}H$, $\nabla^2_{xp}H$, $\nabla^2_{x \mu} H$, $\nabla^2_{pp}H$, and $\nabla^2_{p \mu}H$ exist. Moreover, despite not specifying a measure on $\mathcal{P}_2(\mathbb{R}^d)$, we say that $\nabla^2_{xx}H$, $\nabla^2_{xp}H$, and $\nabla^2_{pp}H$ have finite $L^{\infty}$ norms on their respective domains in the sense that:
\begin{align*}
    \|\nabla^2_{xx} H\|_{\infty} &:= \sup_{\mu \in \mathcal{P}_2(\mathbb{R}^d)} \|\nabla^2_{xx} H(\cdot, \cdot, \mu)\|_{L^{\infty}(\mathbb{R}^d \times \mathbb{R}^d; \mathbb{R}^{d \times d})} < \infty, \\
    \|\nabla^2_{xp} H\|_{\infty} &:= \sup_{\mu \in \mathcal{P}_2(\mathbb{R}^d)} \|\nabla^2_{xp} H(\cdot, \cdot, \mu)\|_{L^{\infty}(\mathbb{R}^d \times \mathbb{R}^d; \mathbb{R}^{d \times d})} < \infty, \\
    \|\nabla^2_{pp} H\|_{\infty} &:= \sup_{\mu \in \mathcal{P}_2(\mathbb{R}^d)} \|\nabla^2_{pp} H(\cdot, \cdot, \mu)\|_{L^{\infty}(\mathbb{R}^d \times \mathbb{R}^d; \mathbb{R}^{d \times d})} < \infty.
\end{align*}
$\nabla_x H$ and $\nabla_p H$ are Lipschitz in the measure argument with Lipschitz constants $\|\nabla^2_{x \mu} H\|_{\infty}$ and $\|\nabla^2_{p \mu} H\|_{\infty}$: for all $\mu_1, \mu_2 \in \mathcal{P}_2(\mathbb{R}^d)$ and $x, p \in \mathbb{R}^d$,
\begin{align*}
    &|\nabla_x H(x, p, \mu_1) - \nabla_x H(x, p, \mu_2)| \leq \|\nabla^2_{x \mu} H\|_{\infty} W_1(\mu_1, \mu_2), \\
    &|\nabla_p H(x, p, \mu_1) - \nabla_p H(x, p, \mu_2)| \leq \|\nabla^2_{p \mu} H\|_{\infty} W_1(\mu_1, \mu_2).
\end{align*}
Finally, we borrow these assumptions from \cite[(2.5), (2.6)]{Meszaros2024}: $H(\cdot, \cdot, \mu) \in C^{1, 1}_{loc}(\mathbb{R}^d \times \mathbb{R}^d)$, uniformly in $\mu$, and $H(x, p, \cdot)$ is continuous in $\mu$ with respect to $W_1$, locally uniformly on $\mathbb{R}^d \times \mathbb{R}^d$.
\end{assumption}
\begin{assumption}[Convexity of $H$]\cite[(2.7)]{Meszaros2024} \label{convexity_of_H}
    $H$ is uniformly convex in $p$: there exists some $c_0 > 0$ such that $\nabla^2_{pp} H \succcurlyeq c_0I_d$. Also, for each $p \in \mathbb{R}^d$, there exists a constant $C(p)$ such that for all $\mu \in \mathcal{P}_2(\mathbb{R}^d)$, $|\nabla_p H(0, p, \mu)| \leq C(p)$.
\end{assumption}

\begin{assumption}[Regularity of $g$] \label{terminal_cost_assumption_2nd_derivatives}
The derivatives $\nabla_\mu g$, $\nabla^2_{xx}g$, and $\nabla^2_{x \mu} g$ exist. Moreover, despite not specifying a measure on $\mathcal{P}_2(\mathbb{R}^d)$, we say that $\nabla^2_{xx}g$ has finite $L^{\infty}$ norm in the sense that:
\begin{equation*}
    \|\nabla^2_{xx}g\|_{\infty} := \sup_{\mu \in \mathcal{P}_2(\mathbb{R}^d)} \|\nabla^2_{xx} g(\cdot, \mu)\|_{L^{\infty}(\mathbb{R}^d; \mathbb{R}^{d \times d})} < \infty.
\end{equation*}
Again, we borrow these assumptions from \cite{Meszaros2024}: $\nabla_x g$ is Lipschitz in the measure argument with respect to $W_1$, and we denote its Lipschitz constant by $\|\nabla^2_{x \mu} g\|_{\infty}$: for all $\mu_1, \mu_2 \in \mathcal{P}_2(\mathbb{R}^d)$ and $x \in \mathbb{R}^d$,
\begin{align*}
    |\nabla_x g(x, \mu_1) - \nabla_x g(x, \mu_2)| \leq \|\nabla^2_{x \mu} g\|_{\infty} W_1(\mu_1, \mu_2).
\end{align*}
Finally, $g(\cdot, \mu) \in C^1_{loc}(\mathbb{R}^d)$ uniformly in $\mu$, and $g(x, \cdot)$ is continuous with respect to $W_1$, locally uniformly in $\mathbb{R}^d$.
\end{assumption}

\begin{assumption} \label{initial_distribution}
The initial condition $m_0$ is an element of $L^2(\Omega, \mathcal{F}_0, \mathbb{P})$. 
\end{assumption}

\begin{assumption} \label{stability_wrt_initial_condition}
The McKean-Vlasov FBSDE \eqref{eq:fbsde} is stable with respect to the initial condition. To be precise, let $\xi, \tilde{\xi} \in L^2(\Omega, \mathcal{F}_0, \mathbb{P})$ have laws $m_0, \tilde{m}_0$ satisfying Assumption \ref{initial_distribution}. Let the solution to Equation \eqref{eq:fbsde} with initial condition $\xi \sim m_0$ be denoted as $(X, Y, Z)$; similarly, let $(\tilde{X}, \tilde{Y}, \tilde{Z})$ denote the solution to Equation \eqref{eq:fbsde} but with initial condition $\tilde{\xi} \sim \tilde{m_0}$. We assume that for some constant $C$ depending only on $T$ and the Lipschitz constants of $\nabla_p H$ and $\nabla_x H$,
\begin{equation}
    \mathbb{E} \bigg[\sup_{t \in [0, T]} \big\{ |X_t - \tilde{X}_t|^2\big\}\bigg] \leq C \mathbb{E} \big[|\xi - \tilde{\xi}|^2\big].
\end{equation}
This assumption is satisfied by FBSDEs covered by certain conditions sufficient for the well-posedness of the FBSDE associated with the MFG, such as displacement monotonicity \cite[Theorem 4.5]{Meszaros2024}. We would like to mention that the specific structure that displacement monotonicity is not the key to our results, as any other condition that implies the property of stability with respect to the initial condition could replace displacement monotonicity without changing our results.
\end{assumption}

\begin{assumption} \label{mfg_wellposedness}
For all $\beta \geq 0$, the MFG \eqref{eq:mfg} is well-posed with solution $(u^{\beta}, \rho^{\beta})$ in the sense of Definition \ref{definition_of_solution}. This assumption will be in force for the rest of the paper, even when we do not say so explicitly.
\end{assumption}

\begin{remark}
We present two classes of Hamiltonians that satisfy Assumptions \ref{hamiltonian_assumption_2nd_derivatives} and \ref{convexity_of_H}.
\begin{enumerate}[itemsep = 3 pt]
\item A quite general class of Hamiltonians is given by the following. Let $F, \gamma_1, \gamma_2: \mathbb{R}^d \mapsto \mathbb{R}$ and $U_1, U_2: \mathcal{P}_2(\mathbb{R}^d) \mapsto \mathbb{R}$. Then
\begin{equation*}
    H(x, p, \mu) = F(x) + \gamma_2(p)+ \gamma_1(p)U_1(\mu) + U_2(\mu)
\end{equation*}
satisfies Assumptions \ref{hamiltonian_assumption_2nd_derivatives} and \ref{convexity_of_H} if $\nabla^2 F$, $U_1$, $(\nabla^2 \gamma_1) U$, and $\nabla \gamma_1 \nabla_{\mu} U_1$ are bounded, and if there exists $C, c > 0$ such that $c \preccurlyeq \nabla^2 \gamma_2 \preccurlyeq C$ and $c \preccurlyeq \nabla^2 \gamma_1 U_1 \preccurlyeq C$.

\item The following Hamiltonian is an example that satisfies Assumptions \ref{hamiltonian_assumption_2nd_derivatives} and \ref{convexity_of_H} due to the addition of a large enough quadratic. Let $\Gamma > 0$, $F: \mathbb{R} \times \mathcal{P}_2(\mathbb{R}^d) \mapsto \mathbb{R}$, $\gamma: \mathbb{R}^d \mapsto \mathbb{R}$, and
\begin{equation*}
    H(x,p,\mu) = \Gamma |p|^2 + \gamma(p) F(x, \mu).
\end{equation*}
Then Assumptions \ref{hamiltonian_assumption_2nd_derivatives} and \ref{convexity_of_H} are satisfied if $\nabla \gamma$, $\nabla^2 \gamma \cdot F$, $F(0, \cdot)$, $\nabla_x F$, $\nabla^2_{xx} F$, $\nabla_{\mu} F$ are bounded and if $\Gamma$ is large enough: $\Gamma > \tfrac{1}{2}\inf \{ \|\nabla^2 \gamma(p)\| F(x, \mu)\}$. 

\item Consider the one-dimensional example 
\begin{equation*}
    H(x, p, \mu) = f(x) + xp +  (p - U(\mu))^2
\end{equation*}
for $f: \mathbb{R} \mapsto \mathbb{R}$, $U: \mathcal{P}_2(\mathbb{R}) \mapsto \mathbb{R}$. If $\nabla^2 f, \nabla_{\mu} U$ are bounded, then $H$ satisfies Assumptions \ref{hamiltonian_assumption_2nd_derivatives} and \ref{convexity_of_H}, and if there is some $C(p)$ such that $|p - U(\mu)| \leq C(p)$ for each $\mu$. This example is inspired by \cite{trading_and_wealth_evolution}, despite its Hamiltonian being time-dependent. While our assumptions do not allow for time-varying Hamiltonians, we expect our results to extend to them under suitable regularity conditions.

\item For a function $F: \mathbb{R} \times \mathbb{R} \mapsto \mathbb{R}$ and a mollifier $\varphi \in C^2_{bc}(\mathbb{R}^d)$ (i.e., its derivatives up to the second order are bounded and compactly supported), $\eta > 0$, and $q > 2$, define
\begin{equation*}
    H(x, p, \mu) =  \Gamma |p|^2 + \frac{\gamma(p)}{|(\varphi \ast \mu)(x) + \eta|^q} - F(x, (\varphi \ast \mu)(x)).
\end{equation*}
If $\gamma, \nabla \gamma, \nabla_x F, \nabla_m F$ are bounded, and if $\Gamma$ is large enough, then Assumptions \ref{hamiltonian_assumption_2nd_derivatives} and \ref{convexity_of_H} are satisfied. This example is motivated by non-separable Hamiltonians from MFGs modeling congestion \cite{AP2018}. The idea behind the convolution of $\mu$ with $\varphi$ is to transform the Hamiltonian that is usually encountered in MFGs with congestion from a local one into a nonlocal one. From a modeling perspective, it should be interpreted as agents not only taking into account the density of agents at their current position, but also the density of agents in some compact set around the agent's position. Ideally we would allow $\gamma$ to be unbounded, such as setting $\gamma(p) = |p|^r/r$ as in \cite{AP2018}, but this would violate the assumption that $\|\nabla^2_{p \mu} H\|_{\infty} < \infty$ as well as the assumption of uniform convexity if $r \neq 2$.

\end{enumerate}
\end{remark}

\quad Our assumptions largely agree with those in \cite{Bansil2024b} (namely, Assumptions 2.6(1) and 2.7(1) therein). To reiterate, \cite{Bansil2024b} is, to the best of our knowledge, the work with the least restrictive regularity assumptions that guarantee well-posedness of the master equation.

\quad As discussed in \cite[Section 3]{Meszaros2024}, Assumption \ref{mfg_wellposedness} can be satisfied by MFGs that do not necessarily possess displacement monotone $H$ or $g$. Indeed, we allow any $H$ and $g$ that satisfy any conditions that are sufficient to guarantee well-posedness (and do not contradict Assumptions \ref{hamiltonian_assumption_2nd_derivatives}--\ref{initial_distribution}). See, for example, \cite{Graber2023, Mou2024} for additional conditions sufficient for well-posedness that are beyond the Lasry-Lions and displacement monotonicity conditions.

\begin{remark}[Comparison with \cite{Tang2023b}] \label{rk:comparison}
We compare our assumptions to those of \cite{Tang2023b} in greater detail. In terms of regularity, their condition (H1') on the $C^2$ norm of the coupling says that the coupling must have bounded first and second derivatives, which is slightly stronger than ours (we allow for $\nabla_x H$ to be unbounded). 
Their condition (H2'), that the second-order derivatives of $H$ are locally bounded, is of course not as strong as uniform boundedness. 
Their condition (H3') is stronger than ours, which requires that the $C^2$ norm of the terminal cost is bounded, while we only require that $g$ be Lipschitz in the measure argument with respect to $W_1$.
Finally, (H4') and (H4'') do not seem to be directly comparable to the regularity of L-derivatives.
However, observe that if the condition (in H4' and H4''):
\begin{equation*}
    \sup_{x \in \mathbb{T}^d} |f(x, \mu') - f(x, \mu)| \lesssim 
    \left(\int_{\mathbb{T}^d} (f(x, \mu') - f(x,\mu)) d(\mu' - \mu)(x) \right)^{\frac{1}{2}},
\end{equation*}
were replaced with
\begin{equation*}
    \sup_{x \in \mathbb{T}^d} |f(x, \mu') - f(x, \mu)| \lesssim
    \int_{\mathbb{T}^d} (f(x, \mu') - f(x,\mu)) d(\mu' - \mu)(x),
\end{equation*}

then they would have achieved the same convergence rate as we did.

\quad Moreover, we did not find necessary to assume the monotonicity conditions that \cite{Tang2023b} imposed on their Hamiltonian and coupling. 
For instance, we could impose any of the four conditions in \cite{Graber2023} that guarantee well-posedness of the MFG system, 
while \cite{Tang2023b} did require the Lasry-Lions monotonicity condition for their proof \footnote{The Lasry-Lions monotonicity condition is needed in their equations (6.10) and (6.11).}.
Most importantly, \cite{Tang2023b} relied on the separable Hamiltonian structure, and their proof technique seems difficult to adapt to the non-separable case on account of the dual equation technique \cite{Lin2001}.
\end{remark}

\subsection{Problem Formulation}

Let $(\Omega, \mathcal{F}, \mathbb{F}, \mathbb{P})$ be a filtered probability space for some time horizon $T > 0$, where the filtration $\mathbb{F}$ is generated by a standard $d$-dimensional Brownian motion $\{B_t\}_{t=0}^T$. For a sequence of measures $\{\mu_s\}_{s=0}^T \subseteq \mathcal{P}(\mathbb{R}^d)$, a representative agent takes a path $\{X^{\alpha, \mu}_s\}_{s=t}^T$ through $\mathbb{R}^d$, which satisfies the SDE:
\begin{equation*}
\begin{cases}
    dX^{\alpha, \mu}_s = \alpha_s(X^{\alpha, \mu}_s, \mu_s) ds + \beta dB_s & \text{for }  s \in (t, T],\\
    X^{\alpha, \mu}_t = x,
\end{cases}
\end{equation*}
for some initial position $x \in \mathbb{R}^d$, initial time $t \in [0, T]$, and adapted stochastic process $\{\alpha_s\}_{s=t}^T$ (referred to as the control). Its goal is to minimize the following cost functional $J$:
\begin{equation*}
    J(t, x; \alpha) = \mathbb{E} \bigg[ \int_t^T L(X^{\alpha, \mu}_s, \alpha_s, \mu_s) ds + g(X^{\alpha, \mu}_T, \mu_T) \bigg| X^{\alpha, \mu}_t = x\bigg],
\end{equation*}
over all controls $\alpha$. Here, the Lagrangian $L$ is a running cost function that is the Legendre transform of $H$ in the second variable, $g$ is a terminal cost function, and $\beta \geq 0$ is the idiosyncratic noise intensity faced by each player. Let $u^{\beta}(t, x) = \inf_{\alpha} J(t, x; \alpha)$. From classical optimal control, if all players play optimally in the sense of a Nash equilibrium, then at time $s$ and position $x$, each player chooses the action $\alpha: [0, T] \times \mathbb{R}^d \times \mathcal{P}(\mathbb{R}^d) \mapsto \mathbb{R}^d$, defined by
\begin{equation*}\label{eq:optimal_action}
    \alpha_t(x, \rho) = \nabla_p H(x, -\nabla u^{\beta}(t, x), \rho),
\end{equation*}
and we can set the sequence of measures $\mu$ to be $\rho^{\beta}$, the solution to the Fokker-Planck equation in Equation \eqref{eq:mfg}. As both $\alpha$ and $\mu$ are fixed, and since our focus is on the dependence of $u^{\beta}$ on $\beta$, we replace $\alpha$ and $\mu$ in $X^{\alpha, \mu}$ by the noise intensity $\beta$. We use $X^{\beta}$ to refer to the stochastic process representing the path of the agent, and we use $\rho^{\beta}_s$ to refer to the law of $X^{\beta}_s$. Moreover, the pair $(u^{\beta}, \rho^{\beta})$ solves the coupled PDEs \eqref{eq:mfg} in the following sense:

\begin{definition}\cite[Definition 3.2]{Meszaros2024}\label{definition_of_solution}
We say that $(u^{\beta}, \rho^{\beta})$ is a solution to the MFG system \ref{eq:mfg} if
\begin{enumerate}[itemsep = 3 pt]
\item for all $t \in [0, T]$, the Lipschitz constant of $u^{\beta}(t, \cdot)$ restricted to any compact set is finite, $u^{\beta}$ is a viscosity solution to the HJB equation, and $\nabla^2_{xx}u^{\beta} \in L^{\infty}([0, T] \times \mathbb{R}^d; \mathbb{R}^{d \times d})$.
\item $\rho^{\beta}_{\cdot}: [0, T] \to (\mathcal{P}_1(\mathbb{R}^d), W_1)$ is continuous, and $\rho^{\beta}$ solves the Fokker-Planck equation in the distributional sense: for all test functions $\varphi \in C^{\infty}_c([0, T] \times \mathbb{R}^d)$, the following equation holds:
\begin{equation*}
    \int_0^T \int_{\mathbb{R}^d} \bigg[ -\partial_t \varphi + \langle \nabla \varphi, \nabla_p H(x, -\nabla u^{\beta}(t, x), \rho^{\beta}_t)\rangle \bigg] d\rho^{\beta}_t(x) dt = \frac{\beta^2}{2} \int_0^T \int_{\mathbb{R}^d}\Delta \varphi d \rho^{\beta}_t(x) dt.
\end{equation*}
Moreover, $\rho_0^\beta$ satisfies the initial condition: 
\begin{equation*}
    \int_{\mathbb{R}^d} \varphi(0, x) dm_0(x) = \int_{\mathbb{R}^d} \varphi(0, x) d \rho^{\beta}_0(x).
\end{equation*}
\end{enumerate}
\end{definition}

\quad It is worth noting that from \cite[Lemma 3.4]{Meszaros2024}, there is an estimate for $\nabla^2_{xx} u^{\beta}$ which is uniform in $\beta$: $K:= \sup_{\beta > 0} \|\nabla^2_{xx} u^{\beta}\|_{L^{\infty}([0, T] \times \mathbb{R}^d; \mathbb{R}^{d \times d})} < \infty$. Moreover, according to \cite[Theorem 4.1]{Meszaros2024}, under Assumption \ref{convexity_of_H}, the MFG \eqref{eq:mfg} has an FBSDE representation:
\begin{equation}\label{eq:fbsde}
\begin{cases}
dX^{\beta}_t =\nabla_p H(X^{\beta}_t, Y^{\beta}_t, \rho^{\beta}_t) dt + \beta dB_t, \\
dY^{\beta}_t = -\nabla_x H(X^{\beta}_t, Y^{\beta}_t, \rho^{\beta}_t) dt + \beta Z^{\beta}_t dB_t, \\
X^{\beta}_0 \sim m_0, \quad Y^{\beta}_T = -\nabla_x g(X^{\beta}_T, \rho^{\beta}_T),
\end{cases}
\end{equation}
and it has a strong solution. The first SDE describes the state dynamics of an agent playing in a Nash equilibrium. The key to our analysis is the fact that $\nabla u^{\beta}: [0, T] \times \mathbb{R}^d \to \mathbb{R}^d$ is a decoupling field for the FBSDE \eqref{eq:fbsde} in the sense that:
\begin{equation*}\label{eq:decoupling_field}
    Y^{\beta}_t = -\nabla u^{\beta}(t, X^{\beta}_t) \quad \mbox{for almost every } t \in [0, T].
\end{equation*}
The meaning of $\nabla u^{\beta}$ to the FBSDE is that the SDE for $Y^{\beta}$ is solved by the gradient of the value function evaluated along the trajectory of a typical agent playing in a Nash equilibrium. Instead of relying on PDE methods such as the dual equation as \cite{Tang2023b} did, we use a probabilistic approach that hinges on stability properties of the FBSDE \eqref{eq:fbsde} to derive our main result.

\section{Convergence of $\rho^{\beta}$ and $\nabla u^{\beta}$} \label{sc3}

\quad The difficulty of analyzing MFGs with a non-separable Hamiltonian is that without the assumption of separability \eqref{eq:separability}, we can no longer analyze each equation separately. However, as discussed previously, we can use the convenient property of the FBSDE \eqref{eq:fbsde} having $-\nabla u^{\beta}$ as a decoupling field in the sense of Equation \eqref{eq:decoupling_field}. By substituting Equation \eqref{eq:decoupling_field} into the SDE for $X^{\beta}$, we can analyze it separately from the SDE for $Y^{\beta}$. Furthermore, because $\rho^{\beta}_t = \operatorname{Law}(X^{\beta}_t)$ for all $t \in [0, T]$, we can get a convergence rate for $\rho^{\beta}$ to $\rho$ in $L^2$, and hence, in $W_2$. 

\quad Before we state the results, we define the following FBSDE system for $t_0 < T$ and $\zeta \in L^2(\Omega, \mathcal{F}_{t_0}; \mathbb{P}; \mathbb{R}^d)$:
\begin{equation}\label{eq:fbsde-arbitrary-initial-time}
\begin{cases}
dX^{\beta}_t = \nabla_p H(X^{\beta}_t, Y^{\beta}_t, \rho^{\beta}_t) dt + \beta dB_t, \\
dY^{\beta}_t = - \nabla_x H(X^{\beta}_t, Y^{\beta}_t, \rho^{\beta}_t) dt + \beta Z^{\beta}_t dB_t, \\
X^{\beta}_{t_0} = \zeta , \quad Y^{\beta}_T = -\nabla_x g(X^{\beta}_T, \rho^{\beta}_T).
\end{cases}
\end{equation}
Its only differences compared to \eqref{eq:fbsde} is that the initial time $t_0$ is not necessarily $0$ and that the initial condition does not need to have the law $m_0$. To lighten notation, when $\beta = 0$, we denote by $X_t = X^0_t$ to be the solution to the FBSDE \eqref{eq:fbsde} or \eqref{eq:fbsde-arbitrary-initial-time}. For a bounded set $U \subseteq \mathbb{R}^d$, we write $A(t_0, T^{\prime}; U) \in \mathcal{F}_{T^{\prime}}$ to be the event that $X$ does not exit $U$ between $[t_0, T^{\prime}]$, i.e.
\begin{equation*}
    A(t_0, T^{\prime}; U) = \{X_s \in U \text{ for all } s \in [t_0, T^{\prime}]\}.
\end{equation*}

\begin{lemma}\label{supersized_set}
    Let $t_0 \in [0, T)$ and $\zeta \in L^{\infty}(\Omega, \mathcal{F}_{t_0}, \mathbb{P}; \mathbb{R}^d)$. If $X$ is the solution to Equation \eqref{eq:fbsde-arbitrary-initial-time} with $\beta = 0$ and initial condition $\zeta$, then there exists a bounded set $U$ such that $A(t_0, T; U)$ has probability $1$. It is worth mentioning for future reference the obvious corollary that for any set $\tilde{U}$ containing $U$, $A(t_0, T; \tilde{U})$ also has probability $1$.
\end{lemma}
\begin{proof}
By a standard argument using Gr\"{o}nwall's inequality, $K < \infty$, the definition of $X$, and $\nabla u$ being a decoupling field,
\begin{equation*}
    \sup_{t \in [0, T]} \mathbb{E} \big[ |X_t|^2 \big] < \infty.
\end{equation*}
It follows that $(x, p, t) \mapsto \nabla_p H(x, p, \rho_t)$ has a Lipschitz constant uniform in $t$. By classical ODE theory, $X_{\cdot}: [t_0, T] \mapsto \mathbb{R}^d$ is an element of $C^1([0, T]; \mathbb{R}^d)$ for all $\omega \in \Omega$. Since the law of $\zeta$ is compactly supported, there exists some bounded set $U$ such that for all $t \in [t_0, T]$, $X_t \in U$ for all $\omega$.
\end{proof}

\begin{lemma}\label{convergence_in_L2_of_Xbeta_by_beta_grad_ubeta}
Let $\beta \geq 0$, $t_0 \in [0, T)$, and $\zeta \in L^{\infty}(\Omega, \mathcal{F}_{t_0}, \mathbb{P}; \mathbb{R}^d)$. Consider the uncoupled state variable dynamics:
\begin{equation}\label{eq:uncoupled_state_dynamics}
\begin{cases}
    dX^{\beta}_t = \nabla_p H(X^{\beta}_t, -\nabla u^{\beta}(t, X^{\beta}_t), \rho^{\beta}_t) dt + \beta dB_t, \\
    X^{\beta}_{t_0} = \zeta.
\end{cases}
\end{equation}
Under Assumptions \ref{hamiltonian_assumption_2nd_derivatives} -- \ref{terminal_cost_assumption_2nd_derivatives}, there exists a constant $C = C(H, g, T)$, where the dependence on $H$ and $g$ is only through the $L^{\infty}$ norms of their second-order derivatives, such that for all $T^{\prime} \in [t_0, T]$,
\begin{equation}
    \sup_{t \in [t_0, T^{\prime}]} \mathbb{E}[|X^{\beta}_t - X_t|^2] \leq C \left\{ \beta^2 + \int_{t_0}^{T^{\prime}} \|\nabla u^{\beta}(s, \cdot) - \nabla u(s, \cdot) \|^2_{L^{\infty}(U)}ds \right\}.
\end{equation}
\end{lemma}
\begin{proof}
Firstly, suppose $U$ is defined as above the statement of the lemma: as a bounded set such that for all $t \in [t_0, T]$, $X_t \in U$. For any $\beta \geq 0$ and $s \in [t_0, T]$, 
\begin{equation*}
    \|\nabla u^{\beta}(s, \cdot) - \nabla u(s, \cdot)\|_{L^{\infty}(U)} < \infty
\end{equation*}
because $\nabla u^{\beta}(s, \cdot)$ is locally Lipschitz continuous in $x$, on account of $K < \infty$. Now, note that the equation \eqref{eq:uncoupled_state_dynamics} is the SDE satisfied by a solution to the $X$ component of the FBSDE \eqref{eq:fbsde-arbitrary-initial-time}, the existence of which is guaranteed by \cite[Theorem 4.1]{Meszaros2024} and Assumption \ref{convexity_of_H}. Using Assumptions \ref{hamiltonian_assumption_2nd_derivatives} and \ref{terminal_cost_assumption_2nd_derivatives}, as well as $K < \infty$, we have that for all $s \in [t_0, T^{\prime}]$:
\begin{equation}\label{eq:difference_in_drift_for_Xbeta_X}
\begin{aligned}
    \mathbb{E} [ |\nabla_p &H(X^{\beta}_s, -\nabla u^{\beta}(s, X^{\beta}_s), \rho^{\beta}_s) - \nabla_p H(X_s, -\nabla u(s, X_s), \rho_s) ] \\
    &\leq \mathbb{E}[\|\nabla^2_{xp} H\|_{\infty} |X_s^{\beta} - X_s| + \|\nabla^2_{pp} H\|_{\infty} (|\nabla u^{\beta} (s, X^{\beta}_s) - \nabla u^{\beta}(s, X_s)| \\
    & \quad \quad \quad + \|\nabla u^{\beta}(s, \cdot) - \nabla u(s, \cdot)\|_{L^{\infty}(U)})] + \|\nabla^2_{p \mu} H\|_{\infty} W_1(\rho^{\beta}_s, \rho_s) \\
    &\leq C \big\{ \mathbb{E}[|X_s^{\beta} - X_s|] + \|\nabla u^{\beta}(s, \cdot) - \nabla u(s, \cdot)\|_{L^{\infty}(U)} \big\}
\end{aligned}
\end{equation}
where the expectations are finite due to the assumption that $\zeta \in L^{\infty}$. As a result, there exists some constant $C$ depending only on $T$ and the Lipschitz constants of the gradients of $H$ and $g$, such that for all $T^{\prime} \leq T$: 
\begin{equation*}
\begin{split}
    \mathbb{E}&[|X^{\beta}_{T^{\prime}} - X_{T^{\prime}}|^2 \cdot \mathbf{1}_{A(t_0, T^{\prime}; U)}] = \mathbb{E}[|X^{\beta}_{T^{\prime}} - X_{T^{\prime}}|^2]\\
    & \leq C \mathbb{E}\left[ \int_{t_0}^{T^{\prime}} |\nabla_p H(X^{\beta}_s, -\nabla u^{\beta}(s, X^{\beta}_s), \rho^{\beta}_s) - \nabla_p H(X_s, -\nabla u(s, X_s), \rho_s)|^2 \ ds + \beta |B_{T^{\prime}}|^2 \right] \\
    & \leq C \left\{ \beta^2 + \int_{t_0}^{T^{\prime}} \mathbb{E}[|X^{\beta}_s - X_s|^2] + \|\nabla u^{\beta}(s, \cdot) - \nabla u(s, \cdot)\|^2_{L^{\infty}(U)} \ ds \right\}
\end{split}
\end{equation*}
where we used the fact that $A(t_0, T^{\prime}; U)$ has probability $1$ in the first line, the definition of $X^{\beta}$ and $X$ in the second line, and $K < \infty$ and Equation \eqref{eq:difference_in_drift_for_Xbeta_X} in the third line. By Gr\"{o}nwall's inequality, there exists another constant $C = C(H, g, T)$ such that
\begin{equation*}
    \mathbb{E}[|X^{\beta}_{T^{\prime}} - X_{T^{\prime}}|^2] \leq C \left\{\beta^2 + \int_{t_0}^{T^{\prime}} \|\nabla u^{\beta}(s, \cdot) - \nabla u(s, \cdot)\|^2_{L^{\infty}(U)} \ ds \right\}.
\end{equation*}
Using the fact that the right-hand side is non-decreasing in $T^{\prime}$ as well as the uniformity of $C$ in $s \in [t_0, T]$ yields the conclusion.
\end{proof}
\quad The previous result hints that in order to quantify the convergence of $\rho^{\beta}_t$ to $\rho_t$, it suffices to control the convergence of $\{\nabla u^{\beta}(t, \cdot) \}_{\beta > 0}$ to $\nabla u(t, \cdot)$ in $L^{\infty}(U)$. In the following theorem, we obtain a convergence rate of $\mathcal{O}(\beta)$ for $\nabla u^{\beta}$ to $\nabla u$ by using its role as the decoupling field.
\begin{theorem}\label{convergence_of_gradient}
Let $u^{\beta}$ and $u$ be solutions to the value function component of the MFG system \eqref{eq:mfg} with $\beta > 0$ and $\beta = 0$, respectively. Suppose that Assumptions \ref{hamiltonian_assumption_2nd_derivatives} -- \ref{terminal_cost_assumption_2nd_derivatives} hold and that $\zeta \in L^{\infty}(\Omega, \mathcal{F}_{t_0}, \mathbb{P}; \mathbb{R}^d)$. Then there exists a constant $C = C(H, g, T)$, where the dependence on $H$ and $g$ is only through the $L^{\infty}$ bounds on the second-order derivatives of $H$ and $g$, and some bounded set $U$ that is large enough for $A(t_0, T; U)$ to have probability $1$, such that
\begin{equation*}
    \|\nabla u^{\beta} - \nabla u\|_{L^{\infty}([0, T] \times U)} \leq C \beta.
\end{equation*}
\end{theorem}
\begin{proof}
Firstly, due to Lemma \ref{supersized_set}, we know that such a $U$ exists and that the range of $\zeta$ is contained in $U$. Denote by $(X^{\beta}, Y^{\beta}, Z^{\beta})$ and $(X, Y)$ the solutions to the FBSDE \eqref{eq:fbsde-arbitrary-initial-time} with the initial condition $X^{\beta}_{t_0} = X_{t_0} = x$, $t_0 \in [0, T)$, and $\beta > 0$ and $\beta = 0$ respectively. We have $Y^{\beta}_t = -\nabla u^{\beta}(t, X^{\beta}_t)$ and $Y_t = -\nabla u(t, X_t)$ for almost every $t$, since we can take conditional expectation of the $X$ and $Y$ components of Equation \eqref{eq:fbsde-arbitrary-initial-time}, conditioned on the event that $X^{\beta}_{t_0} = x$, as in \cite[(4.13)]{Meszaros2024}. Using the uniform convexity of $H$ in $p$ from Assumption \ref{convexity_of_H}, we have that for $\beta \ge 0$, 
\begin{equation*}
    -\nabla u^{\beta}(t_0, x) = \mathbb{E}[\nabla_x g(X^{\beta}_T, \rho^{\beta}_T)] - \int_{t_0}^T \mathbb{E}[\nabla_x H(X^{\beta}_t, Y^{\beta}_t, \rho^{\beta}_t)] dt.
\end{equation*}
By the triangle inequality, we get:
\begin{equation}\label{eq:mere_computation}
\begin{split}
    \big|\nabla & u^{\beta}(t_0, x) - \nabla u(t_0, x)\big| \leq \mathbb{E}\big[\big|\nabla_x g(X^{\beta}_T, \rho^{\beta}_T) - \nabla_x g(X_T, \rho_T)\big|\big] \\
    &+ \int_{t_0}^T \mathbb{E}\big[\big|\nabla_x H(X^{\beta}_t, -\nabla u^{\beta}(t, X^{\beta}_t), \rho^{\beta}_t) - \nabla_x H(X_t, -\nabla u^{\beta}(t, X_t), \rho_t)\big|\big] dt.
\end{split}
\end{equation}
Using Lemma \ref{convergence_in_L2_of_Xbeta_by_beta_grad_ubeta} and Assumption \ref{terminal_cost_assumption_2nd_derivatives}, we can bound the first term in \eqref{eq:mere_computation} by
\begin{equation}
\label{eq:first_term_bound}
\begin{aligned}
    \mathbb{E}[|\nabla_x &g(X^{\beta}_T, \rho^{\beta}_T) - \nabla_x g(X_T, \rho_T)|^2] \leq \|\nabla^2_{xx}g\|^2_{\infty} \mathbb{E}[|X^{\beta}_T - X_T|^2] + \|\nabla^2_{x \mu} g\|_{\infty} W_1(\rho^{\beta}_T, \rho_T)^2 \\ 
    &\leq C \left\{ \beta^2 + \int_{t_0}^{T} \|\nabla u^{\beta}(s, \cdot) - \nabla u(s, \cdot)\|^2_{L^{\infty}(U)} \ ds \right\}.
\end{aligned}
\end{equation}
By Assumption \ref{hamiltonian_assumption_2nd_derivatives}, there is some constant $C$ depending on the Lipschitz constants of $\nabla_x H$ in $x$, $p$, and $\mu$, as well as on $K$, such that we can bound the second term in \eqref{eq:mere_computation} as:
\begin{equation} \label{eq:lastline}
\begin{aligned}
    \int_{t_0}^T& \mathbb{E}[|\nabla_x H(X^{\beta}_t, -\nabla u^{\beta}(t, X^{\beta}_t), \rho^{\beta}_t) - \nabla_x H(X_t, -\nabla u^{\beta}(t, X_t), \rho_t)|^2] \ dt \\ 
    & \leq \int_{t_0}^T \|\nabla^2_{xx} H\|^2_{\infty}\mathbb{E}[|X^{\beta}_t - X_t|^2] + \|\nabla^2_{x \mu} H\|^2_{\infty} W_1(\rho^{\beta}_t, \rho_t)^2 \\ 
    & \quad \quad + \|\nabla^2_{xp} H\|^2_{\infty} \mathbb{E}[|\nabla u^{\beta}(t, X^{\beta}_t) - \nabla u^{\beta}(t, X_t)|^2 + |\nabla u^{\beta}(t, X_t) - \nabla u(t, X_t)|^2 ] \ dt \\ 
    & \leq C \int_{t_0}^T \mathbb{E}[|X^{\beta}_t - X_t|^2]+ \|\nabla u^{\beta}(t, \cdot) - \nabla u(t, \cdot)\|^2_{L^{\infty}(U)} dt.
\end{aligned}
\end{equation}
By applying Lemma \ref{convergence_in_L2_of_Xbeta_by_beta_grad_ubeta} to $\mathbb{E}[|X^{\beta}_t - X_t|]$, using $K < \infty$, and collecting the time integral of $\mathbb{E}[|X^{\beta}_t - X_t|]$ into a supremum, we can continue from the last line in Equation \eqref{eq:lastline} to get:
\begin{equation}\label{eq:second_term_bound}
\begin{aligned}
    \int_{t_0}^T \mathbb{E}&[|\nabla_x H(X^{\beta}_t, -\nabla u^{\beta}(t, X^{\beta}_t), \rho^{\beta}_t) - \nabla_x H(X_t, -\nabla u^{\beta}(t, X_t), \rho_t)|^2] dt \\
    &\leq C \left\{\sup_{t \in [t_0, T]} \mathbb{E}[|X^{\beta}_t - X_t|^2] + \int_{t_0}^T \|\nabla u^{\beta}(t, \cdot) - \nabla u(t, \cdot)\|^2_{L^{\infty}(U)} \ dt \right\} \\ 
    &\leq C \bigg\{ \beta^2 + \int_{t_0}^T \|\nabla u^{\beta}(t, \cdot) - \nabla u(t, \cdot)\|^2_{L^{\infty}(U)} \ dt \bigg\}.
\end{aligned}
\end{equation}
In the above computations, the value of $C$ may change from line to line but only depends on $T$ and the Lipschitz constants of the relevant gradients of $H$ and $g$. Therefore, after taking the supremum over all $x \in U$ and combining \eqref{eq:first_term_bound} and \eqref{eq:second_term_bound}, Equation \eqref{eq:mere_computation} is bounded by
\begin{equation*}
    \|\nabla u^{\beta}(t_0, \cdot) - \nabla u(t_0, \cdot)\|^2_{L^{\infty}(U)} \leq C \left\{ \beta^2 +  \int_{t_0}^T \|\nabla u^{\beta}(s, \cdot) - \nabla u(s, \cdot)\|^2_{L^{\infty}(U)} \ ds \right \}.
\end{equation*}
Then we apply Gr\"{o}nwall's inequality to find that for another constant $C$ still only depending on $H$, $g$, and $T$, 
\begin{equation*}
    \|\nabla u^{\beta}(t_0, \cdot) - \nabla u(t_0, \cdot)\|_{L^{\infty}(U)} \leq C \beta.
\end{equation*}
Since $C$ does not depend on $t_0$, we conclude that
$\|\nabla u^{\beta} - \nabla u\|_{L^{\infty}([0, T] \times U)} \leq C \beta$.
\end{proof}

\begin{corollary}\label{convergence_of_rho}
Suppose that Assumptions \ref{hamiltonian_assumption_2nd_derivatives} -- \ref{mfg_wellposedness} hold. Then, there exists a constant $C = C(H, g, T, m_0)$, where the dependence on $H$ and $g$ is only through their second derivatives, and whose dependence on $m_0$ is only through its second moment, such that:
\begin{equation}
    \sup_{t \in [0, T]} W_2(\rho^{\beta}_t, \rho_t) \leq C \beta.
\end{equation}
\end{corollary}
\begin{proof}
Let $\varepsilon > 0$ and $\xi \sim m_0$. Since $\xi$ is square-integrable, there exists some random variable $\tilde{\xi}$ whose range is contained in a bounded set $U$ such that $\mathbb{E}[|\xi - \tilde{\xi}|^2] < \mathbb{E}[|\xi|^2] \varepsilon$.

\quad To explain further why such a $\tilde{\xi}$ exists, consider the sequence of functions $\{\xi \cdot \mathbf{1}_{\{|\xi| \leq r\}}\}_{r > 0}$, which converges pointwise to $\xi$ as $r \rightarrow \infty$. By the dominated convergence theorem, we can find $R$ large enough such that for $r \geq R$, 
\begin{equation*}
    \varepsilon > \mathbb{E}[|\xi|^2 \cdot \mathbf{1}_{\{|\xi| \geq r\}}]\geq r^2 \mathbb{P}(|\xi| \geq r).
\end{equation*} 
\quad Take $r$ large enough that $U$ is contained in $B_r$. Define the random variable $\tilde{\xi}$ to be the product of $\xi$ and the indicator function of $B_r$, so that 
\begin{equation*}
    \mathbb{E}[|\zeta - \xi|^2] = \mathbb{E}[|\xi|^2 \colon |\xi| \geq r] \leq \mathbb{P}(|\xi| \geq r) \cdot \mathbb{E}[|\xi|^2] < r^{-2}\mathbb{E}[|\xi|^2]\varepsilon.
\end{equation*} 
\quad Thus, if we define $\tilde{U}$ to be the closure of the union of $B_r$ and $\tilde{\xi} = \xi \cdot \mathbf{1}_{\tilde{U}}$, then $\zeta$ is the desired random variable. From now on, we can take $\tilde{U}$ to be $U$ instead. For $\beta > 0$, denote $(X^{\beta}, Y^{\beta}, Z^{\beta})$ and $(\tilde{X}^{\beta}, \tilde{Y}^{\beta}, \tilde{Z}^{\beta})$ to be the solutions to Equation \eqref{eq:fbsde} with initial conditions $\xi$ and $\tilde{\xi}$ respectively. $(X, Y)$ and $(\tilde{X}, \tilde{Y})$ are the solutions to Equation \eqref{eq:fbsde} with initial conditions $\xi$ and $\tilde{\xi}$, but for $\beta = 0$. With the above observations and using the abbreviation $\|\cdot\|$ for $\|\cdot\|_{L^2(\Omega, \mathcal{F}_t, \mathbb{P})}$, for all $t \in [0, T]$:
\begin{equation}\label{eq:final_eq_of_approx_via_init_condition}
\begin{aligned}
    \|X^{\beta}_t - X_t\|  &\leq \|X^{\beta}_t - \tilde{X}^{\beta}_t\| + \|\tilde{X}^{\beta}_t - \tilde{X}_t\| + \|\tilde{X}_t - X_t\| \\
    & \leq C \mathbb{E}[|\xi - \tilde{\xi}|^2] + \left(\beta^2 + \int_t^T \|\nabla u^{\beta}(s, \cdot) - \nabla u(s, \cdot)\|^2_{L^{\infty}(U)} \ ds \right)^{1/2} \\
    & \leq C \{\mathbb{E}[|\xi|^2]\varepsilon + \beta\} = C(\beta + \varepsilon).
\end{aligned}
\end{equation}
\quad In Equation \eqref{eq:final_eq_of_approx_via_init_condition}, to transition from the first line to the second, we applied Assumption \ref{stability_wrt_initial_condition} to handle the first and third terms, and we applied Lemma \ref{convergence_in_L2_of_Xbeta_by_beta_grad_ubeta} to handle the second term; the use of Lemma \ref{convergence_in_L2_of_Xbeta_by_beta_grad_ubeta} was justified because their initial conditions $\tilde{X}^{\beta}_0$ and $\tilde{X}_0$ were assumed to be bounded. The constant in the second line depended on $H$, $g$, and $T$. To handle the term in the square root in the second line, we applied Theorem \ref{convergence_of_gradient}. As promised, in the third line, the dependence of $C$ on $\xi$ was only through absorbing $\mathbb{E}[|\xi|^2]$ into the constant of the second line. Since $\varepsilon$ is arbitrary, the result follows.
\end{proof}

\quad The convergence result for $\nabla u^{\beta}$ may be slightly unsatisfying due to the metric in which it was stated. Thus, although we will not use this result for the rest of the paper, we briefly comment that Assumption \ref{stability_wrt_initial_condition} and Corollary \ref{convergence_of_rho} enable us to derive a rate of convergence of $O(\beta)$ for $\nabla u^{\beta} \rightarrow \nabla u$ in the $L^2(\rho_t)$ metric.

\begin{corollary}\label{cor:L2_convergence_of_grad_ubeta}
Under Assumptions \ref{hamiltonian_assumption_2nd_derivatives}--\ref{mfg_wellposedness}, $\nabla u^{\beta}$ converges to $\nabla u$ at a rate of $\mathcal{O}(\beta)$ in $L^2(\rho_t)$ and $L^2(\rho^\beta_t)$, uniformly in $t$: 
\begin{equation*}
    \sup_{t \in [0, T]} \bigg\{ \|\nabla u^{\beta}(t, \cdot) - \nabla u(t, \cdot)\|_{L^2(\rho_t)} + \|\nabla u^{\beta}(t, \cdot) - \nabla u(t, \cdot)\|_{L^2(\rho^{\beta}_t)} \bigg\} \leq C \beta,
\end{equation*}
for some constant $C$ depending only on $H, g, T, m_0$.
\end{corollary}
\begin{proof}
The families of random variables $\{\nabla u^{\nu}(t, X_t) \colon \nu \in (0, \beta]\}$, indexed by $t$, are uniformly absolutely continuous in the sense that for all $\varepsilon > 0$, there exists a $\delta > 0$ such that for all $A \in \mathcal{F}$ with $\mathbb{P}(A) < \delta$, we have
\begin{equation*}
    \sup_{\nu \in (0, \beta]} \int_A |\nabla u^{\nu}(t, X_t)|^2 d\mathbb{P} < \varepsilon,
\end{equation*}
since for a constant $C > 0$ independent of $t$ and depending only on $K$ and the supremum over $t$ of $\mathbb{E}[|X_t|^2]$ (which is finite by a standard argument for an SDE whose initial condition is $L^2$ and whose coefficients are Lipschitz in space, uniformly in time, and satisfy a linear growth condition),
\begin{equation*}
    \sup_{t \in [0, T]} \int_A |\nabla u^{\nu}(t, X_t)|^2 d\mathbb{P} \leq 2K^2 \sup_{t \in [0, T]} \mathbb{E}[1 + |X_t|^2 \colon A] \leq C \delta < \varepsilon
\end{equation*}
for $\delta$ small enough. In combination with $\sup_{\nu, t} \mathbb{E}[|\nabla u^{\nu}(t, X_t)|^2] < \infty$, this is equivalent to the uniform integrability of $\{\nabla u^{\nu}(t, X_t)\}_{\nu, t}$ by \cite[Theorem 6.5.1]{Resnick}. Moreover, due to Arzela-Ascoli, $\nabla u^{\nu}$ converges to $\nabla u$ uniformly on compacts, so $\nabla u^{\nu}(t, X_t)$ converges to $\nabla u(t, X_t)$ in probability. By the Vitali convergence theorem (see, for example, \cite[Theorem 6.6.1]{Resnick}), 
\begin{equation}\label{eq:vitali}
    \sup_{t \in [0, T]} \|\nabla u^{\beta}(t, \cdot) - \nabla u(t, \cdot)\|^2_{L^2(\rho_t)} = \lim_{\nu \rightarrow 0^+} \sup_{t \in [0, T]} \|\nabla u^{\beta}(t, \cdot) - \nabla u^{\nu} (t, \cdot)\|^2_{L^2(\rho_t)}
\end{equation}
For a constant $C$ depending on $K$ and on the constant from Assumption \ref{stability_wrt_initial_condition}, we obtain
\begin{equation}\label{eq:gradient_fbsde_stability}
    \sup_{t \in [0, T]} \|\nabla u^{\beta}(t, \cdot) - \nabla u^{\nu} (t, \cdot)\|^2_{L^2(\rho_t)} \leq C \mathbb{E} \bigg[\sup_{t \in [0, T]} |X_t - X^{\nu}_t|^2 + |Y^{\beta}_t - Y^{\nu}_t|^2 \bigg] \leq C |\beta - \nu|^2,
\end{equation}
where the first inequality is just by the triangle inequality, Fubini's theorem, and Holder's inequality. (In fact, we could have written $\rho^{\beta}_t$ instead of $\rho_t$ and $X^{\beta}_t$ instead of $X_t$ in everything above without any changes to the proof, which is how convergence in the $L^2(\rho^\beta_t)$ metric can be proven.) Taking the limit as $\nu \rightarrow 0^+$ in Equation \eqref{eq:gradient_fbsde_stability} and then using Equation \eqref{eq:vitali}, we obtain
$\sup_{t \in [0, T]} \|\nabla u^{\beta}(t, \cdot) - \nabla u(t, \cdot)\|^2_{L^2(\rho_t)} \leq C \beta^2$.
\end{proof}

\section{Convergence of $u^{\beta}$} \label{sc4}
\quad We first present a general stability result concerning Hamilton-Jacobi equations with respect to the supremum norm on compact sets. Similar results exist in the literature but are not directly applicable in our scenario \footnote{Closely related results are \cite[Propositions 1.4 and 2.1]{Souganidis1985}. However, the former proposition assumed that the solution to their HJ equation was bounded and that it lacked a second-order term, and the latter assumed that the Hamiltonian was bounded in space and time. Another potentially applicable result was \cite[Lemma 6.3]{Tang2023b}. But on account of their formulation not allowing for non-separable Hamiltonians, we also cannot directly apply this lemma.}. To informally describe the resul: if two HJ PDEs are well-posed with solutions $u$ and $v$, Hamiltonians $H_1$ and $H_2$, terminal cost functions $g_1$ and $g_2$, and viscosity parameters $0$ and $\nu$, then the maximum difference between $u_1$ and $u_2$ on any compact set in $\mathbb{R}^d$ is bounded by the difference in the coefficients $(H_i, g_i)$, on that compact set, and a $\mathcal{O}(\nu^{1/2})$ term.

\begin{theorem}\label{HJ_PDE_stability}
Let $T> 0$. For $i = 1, 2$, let $H_i: \mathbb{R}^d \times \mathbb{R}^d \times [0, T] \to \mathbb{R}$ and $g_i: \mathbb{R}^d \to \mathbb{R}$ be continuous. Suppose that $H_i =H_i(x,p,t)$ is uniformly convex in the second variable $p\in\mathbb{R}^d$, in the sense that there exist constants $c, C > 0$ such that $c I_d \leq \nabla^2_{pp}H_i \leq C I_d$. Assume that $\nabla_p H$ is Lipschitz in the first argument $x\in \mathbb{R}^d$, uniformly in the second and third arguments, i.e. $\|\nabla^2_{xp} H\|_\infty$ is finite, and additionally assume that $\nabla_p H_i(0, 0, \cdot)$ is bounded. If the equations
\begin{equation}\label{eq:standard_HJ_equation}
\begin{cases}
    -\partial_t u_i + H_i(x, -\nabla u_i(t, x), t) = 0, \\
    u_i(T, x) = g_i(x),
\end{cases} 
\end{equation}
are well-posed in the classical sense for $i = 1, 2$ with solution $u_i: [0, T] \times \mathbb{R}^d \to \mathbb{R}$, then for any compact set $\mathcal{K} \subseteq \mathbb{R}^d$ with diameter $R$ and for any $x \in \mathcal{K}$, there exists a constant $\tilde{R} = C(H_1, H_2, g_1, g_2, T, \mathcal{K})$ growing at most linearly in $R$ such that
\begin{equation}
    |u_1(t, x) - u_2(t, x)| \leq (T - t) \|H_1 - H_2\|_{L^{\infty}(B_{\tilde{R}} \times B_{K(1 + \tilde{R}) + A/c} \times [0, T])} + \|g_1 - g_2\|_{L^{\infty}(B_{\tilde{R}})},
\end{equation}
where $K$ satisfies $|\nabla u_i(t, x)| \leq K(1 + |x|)$ for all $(t, x) \in [0, T] \times \mathbb{R}^d$ and $A = A(H_i, \tilde{R}) > 0$ grows at most linearly in $\tilde{R}$.
\end{theorem}
\begin{proof}
Define the costate variable for $i = 1, 2$ as $p^\ast_i(s) = -\nabla u_i(s, \xi^\ast_i(s))$, where $\xi^\ast_i: [0, T] \mapsto \mathbb{R}^d$ is the solution to the following ODE
\begin{equation}\label{eq:ODE_for_xi}
\begin{cases}
    \partial_s \xi^\ast_i(s) = \nabla_p H_i(\xi^\ast_i(s), p^\ast_i(s), s) & s \in [t, T] \\
    \xi^\ast_i(t) = x.
\end{cases}
\end{equation}
\quad In other words, $\xi^\ast_i$ is the optimally controlled trajectory achieving $u_i$ (c.f. Equation \eqref{eq:fbsde-arbitrary-initial-time} with $\beta = 0$). From the assumption of $\nabla_{xp} H_i$ being bounded, Equation \eqref{eq:ODE_for_xi} is well-posed and defined for all $s \in [t, T]$. Let $|x| \leq R$. Due to $\nabla_p H_i(0, 0, \cdot)$ being bounded on $[0, T]$, the sublinear growth of $p^\ast_i$, and the uniform convexity of $H_i$, we can bound the right-hand side of Equation \eqref{eq:ODE_for_xi} as follows:
\begin{equation}
\begin{split}
    \big|\nabla_p H_i(\xi^\ast_i(s), &p^\ast_i(s), s)\big| \leq \big| \nabla_p H_i(0, 0, s)\big| + \|\nabla_{xp}^2 H_i\|_\infty |\xi^\ast_i(s)| + C_i (1 + |p^\ast_i(s)|) \\
    &\leq \| \nabla_p H_i(0, 0, \cdot)\|_\infty + \|\nabla_{xp}^2 H_i\|_\infty |\xi^\ast_i(s)| + C_i K(1 + |\xi^\ast_i(s)|) \\
    &\leq C^\prime_i (1 + |\xi^\ast_i(s)|).
\end{split}
\end{equation}
for some $C^\prime_i$ depending on $C_i$, $K$, $\nabla^2_{xp} H_i$ and $\|\nabla_p H_i(0, 0, \cdot)\|_\infty$. By a standard application of Gr\"{o}nwall's inequality, for all $s \in [t, T]$,
\begin{equation}
    |\xi^\ast_i(s)| \leq (R + C^\prime_i(T - t))(1 + C^\prime_i e^{C^\prime_i(T - t)}).
\end{equation}
\quad Defining the right hand side as $\tilde{R}$, the costate variable can be controlled as well: for all $s \in [t, T]$: $|p^\ast_i(s)| \leq K(1 + \tilde{R})$. Denote $L_i$ as the Fenchel conjugate of $H_i$ in the $p$-variable. Because $\xi^\ast_2$ is a sub-optimal trajectory for $u_1$ but is optimal for $u_2$, in conjunction with the Hopf-Lax representation of $u_i$, we get the inequality
\begin{equation}\label{eq:one_sided_ineq_u1_u2}
    (u_1 - u_2)(t, x) = (g_1 - g_2)(\xi^\ast_2(T)) + \int_t^T (L_1 - L_2)(\xi^\ast_2(s), \partial_s \xi^\ast_2(s), s) ds.
\end{equation}
\quad We want to control the integrand in Equation \ref{eq:one_sided_ineq_u1_u2} by something depending on $H_1 - H_2$. Because $L_i$ is the Fenchel conjugate of $H_i$ in the $p$-variable, and because the supremum is attained uniquely due to the uniform convexity of $H_i$ by some $p_i = p_i(x, v, s)$ solving the equation $v = \nabla_p H_i(x, p, s)$, it follows that for all $(x, v, s) \in \mathbb{R}^d \times \mathbb{R}^d \times [0, T]$,
\begin{equation}
\begin{split}
    (L_1 - L_2)(x, &v, s) \leq \langle p_1, v \rangle - H_1(x, p_1, s) + H_2(x, p_1, s) - \langle p_1, v \rangle \\
    &= H_2(x, p_1, s) - H_1(x, p_1, s).
\end{split}
\end{equation}
\quad Due to symmetry, the same equation can be proven with the subscripts $i = 1, 2$ swapped. Hence, we can bound $|L_1 - L_2|(x, v, s)$ with $x = \xi^\ast_2(s)$ and $v = \partial_s \xi^\ast_2(s)$:
\begin{equation}\label{eq:lagrangian_bound_by_max_of_hamiltonians}
\begin{split}
    |L_1 - L_2|&(\xi^\ast_2(s), \partial_s \xi^\ast_2(s), s) \leq \max \big\{ \big| H_2(\xi^\ast_2(s), p_1, s) - H_1(\xi^\ast_2(s), p_1, s) \big|, \\
    &\quad \quad \big| H_1(\xi^\ast_2(s), p_2, s) - H_2(\xi^\ast_2(s), p_2, s) \big| \big\} \\
    &\leq \max \big\{ \big| H_2(\xi^\ast_2(s), p_1, s) - H_1(\xi^\ast_2(s), p_1, s) \big|, \|H_1 - H_2\|_{L^\infty(B_{\tilde{R}} \times B_{K(1 + \tilde{R})} \times [0, T])} \big\}.
\end{split}
\end{equation}
\quad The second inequality comes from the fact that $\partial_s \xi^\ast_2(s) = \nabla_p H_2 (\xi^\ast_2(s), p_2, s)$ (see the definition of the optimizer in the Fenchel conjugate and the fact that it is unique, so that $p_2 = p^\ast_2$), using Equation \eqref{eq:ODE_for_xi}, and then using $|\xi^\ast_2| \leq \tilde{R}$ and $p_2 \leq K(1 + \tilde{R})$. 

\quad Now we need to control $H_2(\xi^\ast_2, p_1, s) - H_1(\xi^\ast_2, p_1, s)$; since the only unknown quantity is $p_1$, which does not coincide with $p_1^\ast$ in general, we are done if we can contain $p_1$ in some ball whose radius depends on $\tilde{R}$ for all $s \in [t, T]$. Using the uniform convexity of $H_1$, the fact that $p_1(s)$ is the solution to the equation $\partial_s \xi^\ast_2(s) = \nabla_p H_1(\xi^\ast_2(s), p_1(s), s)$, and Equation \eqref{eq:ODE_for_xi}, we obtain:
\begin{equation}\label{eq:bound_p1_minus_p2}
\begin{split}
    |p_1(s) - &p^\ast_2(s)| \leq c^{-1} \big| \nabla_p H_1(\xi^\ast_2(s), p_1(s), s) - \nabla_p H_1(\xi^\ast_2(s), p^\ast_2(s), s)\big| \\
    &\leq c^{-1} \big| \partial_s \xi^\ast_2(s) - \nabla_p H_1(\xi^\ast_2(s), p^\ast_2(s), s)\big| \\
    &\leq c^{-1} \big| \nabla_p H_2(\xi^\ast_2(s), p^\ast_2(s), s) - \nabla_p H_1(\xi^\ast_2(s), p^\ast_2(s), s)\big|.
\end{split}
\end{equation}
Because $|p^\ast_2(s)| \leq K(1 + \tilde{R})$, $|\xi^\ast_2(s)| \leq \tilde{R}$, and $\nabla_p H_i$ is continuous, we see that the last line is bounded by some constant $A = A(\nabla_p H_i, \tilde{R}, K)$, and so
\begin{equation}
    |p_1(s)| \leq |p^\ast_2(s)| + |p_1(s) - p^\ast_2(s)| \leq K(1 + \tilde{R}) + A/c,
\end{equation}
and due to $|\nabla^2_{xp}H_i|$ and $|\nabla^2_{pp}H_i|$ being bounded, $A$ grows at most linearly in $\tilde{R}$. This yields a bound on $H_2(\xi^\ast_2, p_1, s) - H_1(\xi^\ast_2, p_1, s)$, because
\begin{equation}\label{eq:hamiltonian_bound}
    \big| H_2(\xi^\ast_2(s), p_1(s), s) - H_1(\xi^\ast_2(s), p_1(s), s) \big| \leq \|H_1 - H_2\|_{L^\infty(\bar{B}_{\tilde{R}} \times \bar{B}_{K(1 + \tilde{R}) + A/c} \times [0, T])}.
\end{equation}
\quad Therefore, substituting Equation \eqref{eq:hamiltonian_bound} into Equation \eqref{eq:lagrangian_bound_by_max_of_hamiltonians} and then into the Hopf-Lax representation in Equation \ref{eq:one_sided_ineq_u1_u2}, we obtain:
\begin{equation}
    (u_1 - u_2)(t, x) \leq \|g_1 - g_2\|_{L^\infty(\bar{B}_{\tilde{R}})} + (T - t)\|H_1 - H_2\|_{L^\infty(\bar{B}_{\tilde{R}} \times \bar{B}_{K(1 + \tilde{R}) + A/c} \times [0, T])}.
\end{equation}
The roles of $u_1$ and $u_2$ were symmetrical and can therefore be swapped, which proves the result.
\end{proof}

\quad The following result is well-known (see, for instance, e.g. \cite[Lemma 7.1]{Tang2023b}), but we restate it since we want a precise dependence of its constants on the diameter of the set $\mathcal{K}$.
\begin{proposition}\label{HJ_PDE_stability_VV}
Let $T> 0$ and $\nu \geq 0$. Suppose $\bar{H}: \mathbb{R}^d \times \mathbb{R}^d \times [0, T] \to \mathbb{R}$ and $\bar{g}: \mathbb{R}^d \to \mathbb{R}$ satisfy the same assumptions as $H_i$ and $g_i$ do in Theorem \ref{HJ_PDE_stability}. If the equations
\begin{equation}\label{eq:HJ_equations_for_VV_result}
\begin{cases}
    -\partial_t u + \bar{H}(x, -\nabla u(t, x), t) = 0, \\
    u(T, x) = \bar{g}(x),
\end{cases} 
\quad \quad 
\begin{cases}
    -\partial_t v + \bar{H}(x, -\nabla v(t, x), t) = \nu \Delta v, \\
    v(T, x) = \bar{g}(x),
\end{cases} 
\end{equation}
are well-posed for with solution $u, v: [0, T] \times \mathbb{R}^d \to \mathbb{R}$, then for any compact set $\mathcal{K} \subseteq \mathbb{R}^d$, there exists a constant $C_{\mathcal{K}} = C(\mathcal{K}, \bar{H}, 
\bar{g}, T)$ and another constant $C = C(H, g, T)$ such that
\begin{equation*}
    \|u - v\|_{L^{\infty}([0, T] \times \mathcal{K})} \leq C_{\mathcal{K}} \nu^{1/2}.
\end{equation*}
where $C_{\mathcal{K}}$ grows at most quadratically in $\operatorname{diam}(\mathcal{K}) := \sup\{|x - y| \colon x, y \in \mathcal{K}\}$.
\end{proposition}
\begin{proof}
By modifying the proof of \cite[Lemma 7.1]{Tang2023b}, taking $\Omega = \mathcal{K}$, using the linear-in-$x$ growth of $\nabla u$ and $\nabla v$, and the quadratic-in-$x$ growth of $\partial_t u$ and $\partial_t v$, we find that there exists a constant $C = C(\bar{H}, \bar{g}, T)$ such that for any compact set $\mathcal{K} \subseteq \mathbb{R}^d$,
\begin{equation*}
    \|u - v\|_{L^{\infty}([0, T] \times \mathcal{K})} \leq C(1 + \operatorname{diam}(\mathcal{K})^2) \nu^{1/2}.
\end{equation*}
Technically, we should take the maximum of $\operatorname{diam(\mathcal{K})}$ and its square, but for the sake of simplicity, we opt to omit this.
\end{proof}

\begin{corollary}\label{HJ_stability_wrt_everything}
Let $H_0: \mathbb{R}^d \times \mathbb{R}^d \times [0, T] \mapsto \mathbb{R}$ and $g_0: \mathbb{R}^d \mapsto \mathbb{R}$ satisfy the assumptions in Theorem \ref{HJ_PDE_stability}, and let $\bar{H}$ and $\bar{g}$ satisfy the assumptions in Theorem \ref{HJ_PDE_stability_VV}. Suppose $v_0$ satisfies Equation \ref{eq:standard_HJ_equation} with $(H_0, g_0)$, and suppose $v$ satisfies the second PDE in Equation \ref{eq:HJ_equations_for_VV_result} with $(\bar{H}, \bar{g})$. Then for any compact set $\mathcal{K}$ with diameter $R$, there exists a constant $C_\mathcal{K} = C(R, H_0, \bar{H}, g_0, \bar{g}, T)$ growing at most quadratically in $R$, such that
\begin{equation*}
    \|v_0 - v\|_{L^{\infty}([0, T] \times \mathcal{K})} \leq C_\mathcal{K}\big\{\|H_0 - \bar{H}\|_{L^{\infty}(B_{\tilde{R}} \times B_{K(1 + \tilde{R}) + A/c} \times [0, T])} + \|g_0 - \bar{g}\|_{L^{\infty}(B_{\tilde{R}})} + \nu^{1/2}\big\}
\end{equation*}
where $\tilde{R}$, $A$, and $K$ are as in Theorem \ref{HJ_PDE_stability}.
\end{corollary}
\begin{proof}
This follows from applying the triangle inequality to the results of Theorem \ref{HJ_PDE_stability} and Proposition \ref{HJ_PDE_stability_VV}.
\end{proof}

\quad Now we combine the results in Section 3 with Corollary \ref{HJ_stability_wrt_everything} and make some modifications to the proof of Theorem \ref{HJ_PDE_stability} that are particular to the case of measure-dependent $H$ in order to derive the promised $\mathcal{O}(\beta)$ rate of convergence.

\begin{theorem}\label{main_result}
For $\beta \geq 0$, suppose the MFG system \eqref{eq:mfg} satisfies Assumptions \ref{hamiltonian_assumption_2nd_derivatives}--\ref{mfg_wellposedness}, whose solutions for the HJB equation are denoted $u^{\beta}$ and $u$.
\begin{enumerate}[itemsep = 3 pt]
    \item $u^{\beta}$ converges uniformly to $u$ on compacts at a rate of $\mathcal{O}(\beta)$: for any compact set $\mathcal{K} \subseteq \mathbb{R}^d$, there exists a constant $C = C_\mathcal{K}$ depending on the data $(H, g, T, m_0)$ such that
    \begin{equation}\label{eq:value_func_unif_conv_compacts_beta}
        \|u^{\beta} - u\|_{L^{\infty}([0, T] \times \mathcal{K})} \leq C_\mathcal{K}\beta,
    \end{equation}
    when $\beta$ is small enough.
    \item If additionally $H$ is Lipschitz in the measure argument with respect to $W_1$, then for some constant $C_\mathcal{K}$ depending on the data $(H, g, T, m_0)$ that grows at most quadratically in $\operatorname{diam}(\mathcal{K})$, for all $\beta \geq 0$, 
    \begin{equation*}
        \|u^{\beta} - u\|_{L^{\infty}([0, T] \times \mathcal{K})} \leq C_\mathcal{K}\beta.
    \end{equation*}
\end{enumerate}
\end{theorem}
\begin{proof}
In Theorem \ref{HJ_PDE_stability}, let us specialize to the case where $u_1 = u^\beta$ and $u_2 = u$, so that $H_1 = H(\cdot, \cdot, \rho^\beta_{\cdot})$ and $H_2 = H(\cdot, \cdot, \rho_{\cdot})$. Recall that $(X, Y)$ is the solution to the FBSDE in Equation \ref{eq:fbsde} with $\beta = 0$. Note that $H_1 - H_2$ is jointly continuous on $\bar{B}_{R_1} \times \bar{B}_{R_2} \times [0, T]$ for any $R_1, R_2 > 0$, since for any $(x, p, s)$, $\nabla_x H(\cdot, \cdot, \rho_\cdot)$ is bounded on $\bar{B}_{R_1} \times \bar{B}_{R_2} \times [0, T]$, due to a Taylor expansion around $(0, 0, \rho_0)$, Assumption \ref{hamiltonian_assumption_2nd_derivatives}, and $\{\rho_s\}_{s = 0}^T$ being bounded with respect to $W_2$. For the same reasons, $\nabla_p H(\cdot, \cdot, \rho_\cdot)$ is bounded on  $\bar{B}_{R_1} \times \bar{B}_{R_2} \times [0, T]$. Noting that $H$ is continuous in the measure variable with respect to $W_1$ and that the same argument holds when $\rho$ is replaced with $\rho^\beta$ closes the argument for joint continuity of $H$ on $\bar{B}_{R_1} \times \bar{B}_{R_2} \times [0, T]$. It follows that the difference of $H_1$ and $H_2$ on $\bar{B}_{R_1} \times \bar{B}_{R_2} \times [0, T]$ is achieved at some $(x_\ast, p_\ast, s_\ast) \in \bar{B}_{R_1} \times \bar{B}_{R_2} \times [0, T]$, which we will abbreviate as $(x, p, s)$ for convenience. Thus, for $R_1 = \tilde{R}$ and $R_2 = K(1 + \tilde{R}) + A/c$,
\begin{equation}\begin{split}\label{eq:two_taylor_expansions_H}
   \|H(\cdot, &\cdot, \rho_\cdot) - H(\cdot, \cdot, \rho^\beta_\cdot)\|_{L^\infty(\bar{B}_{R_1} \times \bar{B}_{R_2} \times [0, T])} \\
    &\leq \mathbb{E}\big[\big|\langle \nabla_\mu H(x, p, \rho_s, X_s), X_s - X^\beta_s \rangle \big| \big] + \mathcal{O}(\mathbb{E}[|X_s - X^\beta_s|^2]^{1/2}) \\
    &\leq C \big\{\|\nabla_\mu H(0, 0, \rho_s, \cdot)\|_{L^2(\rho_s)} + \|\nabla^2_{x\mu} H\|_\infty |x| + \|\nabla^2_{p\mu} H\|_\infty |p| \big\} \beta + \mathcal{O}(\beta) \\
    &\leq C\big(1 + R_1 + R_2 \big)\beta + \mathcal{O}(\beta).
\end{split}\end{equation}
\quad The second line is via a Taylor expansion of $H$ in the measure variable, which occurs only when $\rho_t$ and $\rho_t^\beta$ are close enough in $W_2$ - see Corollary \ref{convergence_of_rho}. The third line comes from the same corollary, from Cauchy-Schwartz, and from the Taylor expansion of $\nabla_\mu H$ in $x$ and $p$. The fourth line collects the relevant constants into $C$. 

\quad Now, we bound $K(1 + \tilde{R}) + A/c$ more precisely. Recall that $A$ comes from Equation \eqref{eq:bound_p1_minus_p2}; since $H_2 = H(\cdot, \cdot, \rho_\cdot)$, we can identify $\xi^\ast_2$ in Theorem \ref{HJ_PDE_stability} with $X$ solving Equation \eqref{eq:fbsde-arbitrary-initial-time} with $\beta = 0$ and $t_0 = t$. Continuing from Equation \eqref{eq:bound_p1_minus_p2}, we are able to bound $A$ as follows:
\begin{equation}\label{eq:mfg_bound_on_A}
\begin{split}
    \big| \nabla_p H_2&(\xi^\ast_2(s), p^\ast_2(s), s) - \nabla_p H_1(\xi^\ast_2(s), p^\ast_2(s), s)\big| \\
    &= \big| \nabla_p H(X_s, -\nabla u(s, X_s), \rho^\beta_s) - \nabla_p H(X_s, -\nabla u(s, X_s), \rho_s)\big| \\
    &\leq \|\nabla^2_{p\mu} H\|_\infty \mathbb{E}\big[ \big| X^\beta_s - X_s \big|^2\big]^{1/2} + \mathcal{O}(\mathbb{E}\big[ \big| X^\beta_s - X_s \big|^2\big]^{1/2}) \leq C\beta + \mathcal{O}(\beta).
\end{split}
\end{equation}
Therefore, $K(1 + \tilde{R}) + A/c$ is of at most linear growth in $\tilde{R}$, and
\begin{equation}
    \|H(\cdot, \cdot, \rho^\beta_\cdot) - H(\cdot, \cdot, \rho_\cdot)\|_{L^\infty(\bar{B}_{R_1} \times \bar{B}_{R_2} \times [0, T])} \leq C (1 + \tilde{R}) \beta + \mathcal{O}(\beta).
\end{equation}
Now we want to bound $\|g(\cdot, \rho^\beta_T) - g(\cdot, \rho_T)\|_{L^\infty(B_{\tilde{R}})}$, which is achieved for some $x_\ast \in \bar{B}_{\tilde{R}}$ due to the continuity of $g$ in $x$. From Corollary \ref{convergence_of_rho}, for $\beta$ small enough,
\begin{equation}\begin{split}\label{eq:two_taylor_expansions_g}
    \|g(\cdot, \rho^{\beta}_T) &- g(\cdot, \rho_T)\|_{L^{\infty}(\bar{B}_{\tilde{R}})} = \mathbb{E}[\langle \nabla_{\mu} g(x_\ast, \rho_T, X_T), X^{\beta}_T - X_T \rangle ] + \mathcal{O}(\beta) \\
    & \leq \big\{ \|\nabla_{\mu} g(0, \rho_T, \cdot)\|_{L^1(\rho_T)} + \|\nabla^2_{x \mu} g\|_{\infty} |x_\ast| \big\} \cdot \mathbb{E}\big[\big|X^{\beta}_T - X_T\big|^2\big]^{1/2} + \mathcal{O}(\beta) \\ 
    & \leq C(1 + \tilde{R}) \beta + \mathcal{O}(\beta).
\end{split}\end{equation}
for some $C$ depending on $g$ and $\rho_T$ (and so on $H$, $T$, and $m_0$ by extension). Substituting Equation \ref{eq:mfg_bound_on_A} into Equation \ref{eq:two_taylor_expansions_H}, and then substituting Equation \ref{eq:two_taylor_expansions_g} into the bound provided by Corollary \ref{HJ_stability_wrt_everything}, we obtain:
\begin{equation}
\begin{split}
    \|u^\beta &- u\|_{L^\infty([0, T] \times \mathcal{K})} \leq C_\mathcal{K}\big\{ \beta + \|g(\cdot, \rho^{\beta}_T) - g(\cdot, \rho_T)\|_{L^{\infty}(\bar{B}_{R_1})} \\
    & \quad \quad + T\|H(\cdot, \cdot, \rho^\beta_\cdot) - H(\cdot, \cdot, \rho_\cdot)\|_{L^\infty(\bar{B}_{R_1} \times \bar{B}_{R_2} \times [0, T])} \big\} \leq C_\mathcal{K} \beta + \mathcal{O}(\beta),
\end{split}
\end{equation}
which concludes the proof of (1).

\quad To prove (2), we first apply Corollary \ref{HJ_stability_wrt_everything} with $H_0(x, p, t) = H(x, p, \rho_t)$, $\bar{H}(x, p, t) = H(x, p, \rho^\beta_t)$, $g_0(x) = g(x, \rho_T)$, $\bar{g}(x) = g(x, \rho^\beta_T)$, and $\nu = \beta^2/2$. Abbreviating the difference for $H$ on $L^{\infty}(\bar{B}_{R_1} \times \bar{B}_{R_2} \times [0, T])$ and the difference for $g$ on $L^{\infty}(B_{R_1})$, for a constant $C_\mathcal{K}$ having the same dependencies as the one from Corollary \ref{HJ_stability_wrt_everything}, we have:
\begin{align*}
    \|u^{\beta} &- u\|_{L^{\infty}([0, T] \times \mathcal{K})}\leq C_{\mathcal{K}} \left\{ \|H(\cdot, \cdot, \rho^{\beta}_{\cdot}) - H(\cdot, \cdot, \rho_{\cdot})\|_{\infty} + \|g(\cdot, \rho^{\beta}_T) - g(\cdot, \rho_T)\|_{\infty} + \beta/\sqrt{2} \right\} \\
    & \leq C_{\mathcal{K}} \left\{ \big(\|\nabla_{\mu} H\|_{\infty} + \|\nabla_{\mu} g\|_{\infty} \big) \sup_{t \in [0, T]} W_2(\rho^{\beta}_t, \rho_t) + \beta \right\} \leq C_{\mathcal{K}} \beta,
\end{align*}
where the third inequality used Corollary \ref{convergence_of_rho}. 
\end{proof}
\begin{remark}
    If it is only assumed that $H$ and $g$ are continuous in $W_1$, by the stability of viscosity solutions to the HJB equation, we can conclude that $u^{\beta} \rightarrow u$ uniformly on compacts, albeit without a rate.
\end{remark}
\begin{remark}
If $\|\nabla^2_{\mu \mu} H\|_{\infty}$ and $\|\nabla^2_{\mu \mu} g\|_{\infty}$ are assumed to be finite, as \cite{Bansil2024b} does, then we can derive a stronger result: in the Taylor expansions of Equations \eqref{eq:two_taylor_expansions_H} and Equations \eqref{eq:two_taylor_expansions_g}, we can replace the term of $\mathcal{O}(\beta)$ by $C^\prime \beta^2$, where $C^\prime$ depends on $\|\nabla^2_{\mu \mu} H\|_{\infty}$ and $\|\nabla^2_{\mu \mu} g \|_{\infty}$. Then, 
\begin{equation*}
    \|u^{\beta} - u\|_{L^{\infty}([0, T] \times \mathcal{K})} \leq C_\mathcal{K}\beta + C^\prime\beta^2.
\end{equation*}
\end{remark}

\section{Applications} \label{sc5}
\quad This section provides three applications of our result to
$N$-player games, mean field control, and policy iteration.

\subsection{$N$-player games}
MFGs arise as the limit of $N$-player games as the number of players $N$ increases to infinity. Although it is known in various circumstances \cite{Djete2022, Fischer2017, Lacker2020} that the limit is the MFG equilibrium, finding the convergence rate is a separate and difficult problem. The twin papers \cite{Delarue2019, Delarue2020} seem to comprise the most recent progress on determining the convergence rate. However, their results cannot be directly applied to the $N$-player convergence rate problem if the agents follow deterministic dynamics, because one of their assumptions, namely A.2 in both papers, is that the volatility coefficient $\Sigma$ is non-degenerate \footnote{When we say that the volatility $\Sigma$ is non-degenerate, we mean that its minimum eigenvalue is positive. Moreover, if the minimum eigenvalue of $\Sigma$ is allowed to vanish, then their upper bounds for the distance between the probability distribution of the finite player system and that of the MFG limit become infinite.}. Here we apply Corollary \ref{convergence_of_rho} to approximate the probability flow $\rho_t$ of the first-order MFG by the empirical measures of an $N$-player system with non-degenerate volatility.

\quad To simplify our discussion, we only consider the linear drift $b(t,x,a) = a$. So by \cite[(2.6)]{Delarue2020}, the value functions of all $N$ players, $\{v^{N, i}: [0, T] \times (\mathbb{R}^d)^N \to \mathbb{R}\}_{i=1}^N$, satisfy the $N$-player system of PDEs whose $i$-th component is:
\begin{equation*}\label{eq:N_player_system}
\begin{cases}
    \begin{split}
        \partial_t v^{N, i}(t, x) + & H(x_i, \nabla_x v^{N, i}(t, x), m^N_x) + \frac{1}{2} \sum_{j=1}^N \operatorname{Tr}(\Sigma \Sigma^T \nabla^2_{x_j x_j} v^{N, i}(t, x))  \\
        &- \sum_{j \neq i} \langle \nabla_p H(x_j, \nabla_{x_j} v^{N, j}(t, x), m^N_x), \nabla_{x_j} v^{N, i}(t, x) \rangle = 0,
    \end{split} \\
    v^{N, i}(T, x) = g(x_i, m^N_x),
\end{cases}
\end{equation*}
where $m^N_x: = \frac{1}{N} \sum_{i = 1}^N \delta_{x_i}$ is the empirical measure of $x = (x_1, \ldots, x_N) \in (\mathbb{R}^d)^N$. Specializing to the case of $b(t, x, a) = a$, the $i$-th player's dynamics are:
\begin{equation} \label{eq:Nplayer}
    dX^i_t = \alpha^i_t dt + \Sigma dB^i_t = \nabla_p H(X^i_t, -\nabla u^{\sigma}(t, X^i_t), m^{N, \Sigma}_{X_t}) dt + \Sigma dB^i_t,
\end{equation}
where $\{B^i\}_{i=1}^N$ are independent $d$-dimensional Brownian motions 
and $m^{N, \Sigma}_{X_t}$ is the (random) empirical measure of the $N$-player system \eqref{eq:Nplayer} at time $t \in [0,T]$:
\begin{equation}
\label{eq:empiricalXt}
 m^{N, \Sigma}_{X_t} = \frac{1}{N}\sum_{i=1}^N \delta_{X^i_t}.
\end{equation} 

\begin{corollary} \label{coro:Nplayer}
Let $\rho$ satisfy the Fokker-Planck equation in the MFG \eqref{eq:mfg} with $\beta = 0$. Let $\Sigma = \beta I$, and denote by $m^{N, \beta}_{X_t}$ the empirical measure in \eqref{eq:empiricalXt} corresponding to $\Sigma = \beta I$. Under the assumptions in Corollary \ref{convergence_of_rho}, and under \cite[Assumption A]{Delarue2020} and \cite[Assumption B or B']{Delarue2020}, there exist $C_1 = C_1(H,g,T)$ and $C_2 = C_2(\beta, H,g,T)$ such that for all $t \in [0,T]$,
\begin{equation}
    W_1(\rho_t, m^{N, \beta}_{X_t}) \leq C_1 \beta + C_2 N^{-\frac{1}{d+8}}.
\end{equation}
\end{corollary}
\begin{proof}
Let $\rho^{\beta}_t$ satisfy the Fokker-Planck equation 
in the MFG \eqref{eq:mfg} with $\beta > 0$.
We have:
\begin{equation}
    \label{eq:triangleNplayer}
    W_1(\rho_t, m^{N, \beta}_{X_t}) \leq W_1(\rho_t, \rho_t^{\beta}) + W_1(\rho^\beta_t, m^{N, \beta}_{X_t}).
\end{equation}
By \cite[Theorem 3.1]{Delarue2020}, there is a constant $C_2 = C_2(\beta, H,g,T)$ such that
\begin{equation}
    \label{eq:DLRbd}
    \sup_{t \in [0,T]} W_1(\rho^\beta_t, m^{N, \beta}_{X_t}) \le C_2 N^{-\frac{1}{d+8}}.
\end{equation}
Combining Equations \eqref{eq:triangleNplayer}, \eqref{eq:DLRbd} with Corollary \ref{convergence_of_rho} yields the desired bound.
\end{proof}

\quad As a result of Corollary \ref{coro:Nplayer}, we obtain the population level to approximate the probability flow $\rho_t$ of the first-order MFG via large player system.
Assume that an accuracy of $\varepsilon > 0$ is needed,
i.e., $W_1(\rho_t, m^{N, \beta}_{X_t}) \le \varepsilon$.
Then we set:
\begin{equation}
    C_1 \sigma \asymp \varepsilon \quad \mbox{and} \quad C_2(\sigma) N^{-\frac{1}{d+8}} \asymp \varepsilon.
\end{equation}
Here we assume that $(H,g,T)$ are given, so $C_2$ only depends on $\sigma$.
A close scrutiny of the proofs (in particular, Equations 4.16 and 4.17) in \cite{Delarue2020} indicates that $C_2(\beta)$ blows up (in a rather complicated way), as $\beta \to 0$.
So we first take $\sigma \asymp \varepsilon$,
and then take $N \asymp \left( \varepsilon  C_2^{-1}(\varepsilon) \right)^{-(d+8)}$. That is, it requires at most $N \left( \varepsilon  C_2^{-1}(\varepsilon) \right)^{-(d+8)} \gg \varepsilon^{-(d+8)}$ players to approximate the probability flow of the first-order MFG with accuracy $\varepsilon$.

\subsection{Mean field control}

Next we consider a mean field control problem \cite[Proposition 2.14]{DDJ2023}, where a central planner seeks to control $N$ particles by selecting a control process $\alpha = (\alpha^1, ..., \alpha^N)$ from $\mathcal{A}^N$, which is the set of $(\mathbb{R}^d)^N$-valued, progressively measurable processes, whose definition is immediately below \cite[Equation 2.8]{DDJ2023}. Throughout this subsection, $\beta > 0$ is fixed, and $N$ may vary. The dynamics of the $i$-th particle evolve as:
\begin{equation*}
\begin{cases}
    dX^i_t = \alpha^i_t(X_t^i) dt + \beta  dB^i_t \quad \mbox{for } t \in [t_0, T],\\
    X^i_{t_0} = x^i_{0},
\end{cases}
\end{equation*}
where $\{B^i\}_{i=1}^N$ are independent $d$-dimensional Brownian motions.
Denote the average state of the particles by $\overline{X}^N_t = \frac{1}{N} \sum_{i=1}^N X^i_t$, which satisfies the SDE
\begin{equation}
\begin{cases}
    d\overline{X}^N_t = \frac{1}{N} \sum_{i=1}^N \alpha^i_t dt + \frac{\beta}{\sqrt{N}} d\overline{B}_t \quad \mbox{for } t \in [t_0, T],\\
    \overline{X}^N_{t_0} = \frac{1}{N}\sum_{i=1}^N x^i_0,
\end{cases}
\end{equation}
where $\overline{B}_t = N^{-1/2}\sum_{i = 1}^N B^i_t$ is a $d$-dimensional Brownian motion. The objective of the central planner is to solve the optimization problem in \cite[Equation 2.21]{DDJ2023}:
\begin{equation} \label{eq:VN}
    V^N(t_0, x_0) = \inf_{\alpha} \mathbb{E} \left[ \int_{t_0}^T \frac{1}{N}\sum_{i=1}^N L(\alpha^i_t(X^i_t)) + F(\overline{X}_t) dt + G(\overline{X}_T) \bigg| X_{t_0} = x_0 \right],
\end{equation}
where $F, G: \mathbb{R}^d \to \mathbb{R}$ are assumed to be Lipschitz, $L \in C^2(\mathbb{R}^d)$ satisfies the second-derivative bounds $\frac{1}{C}I \leq \nabla^2 L \leq C I$ for some $C \geq 1$, and 
{$\alpha^i$ is a function of both $\omega$ and $x$ (though the optimizer is a deterministic function)}. An easy argument from \cite{DDJ2023} shows that the optimality in \eqref{eq:VN} is achieved by a deterministic control, and $V^N(t, x) = v^N(t, \overline{m}^N_x)$, where $v^N$ solves the HJ equation:
\begin{equation}
\begin{cases}
    -\partial_t v^N(t, x) + H(-\nabla v^N(t, x)) - F(x) = \frac{\beta^2}{2N} \Delta v^N(t, x),  \\
    v^N(T, x) = G(x), 
\end{cases}
\end{equation}
and where $H(-p)$ is the Legendre transform of $L$. By classical viscosity theory, $v^N$ converges to $v$, which is the solution to the first-order equation:
\begin{equation} \label{eq:vHJ}
\begin{cases}
    -\partial_t v(t, x) + H(-\nabla v(t, x)) - F(x) = 0,  \\
    v(T, x) = G(x).
\end{cases}
\end{equation}
Furthermore, $\sup_{[0,T] \times \mathbb{R}^d}|v^N - v| = \mathcal{O}(N^{-\frac{1}{2}})$.

\quad Let $\mu^N_t := \operatorname{Law}(\overline{X}^N_t)$ be the probability density of the average state $\overline{X}^N_t$. The following result specifies the limit of $\mu^N_t$, as $N \to \infty$.
\begin{corollary}
Let the aforementioned assumptions and those in Corollary \ref{convergence_of_rho} hold. Let $\{X^i_0\}_{i=1}^N$ be independent and identically distributed according to $m_0$ with bounded support \footnote{For $d = 1$, the assumption of bounded support can be removed, and $W_1(\mu^N_t, \mu_t) \le C/\sqrt{N}$ for some $C > 0$ (independent of $N$). This is because the first term in the last inequality of \eqref{eq:mu3} is bounded by $C/\sqrt{N}$; see the discussion after \cite[Theorem 3.4]{TT23} or \cite{Rio09}.}, and covariance matrix $\Sigma$. Then for all $t \in [0,T]$, $\mu_t^N$ converges to $\mu_t$ in $W_1$, where $\mu_t$ is the solution to the equation:
\begin{equation}
\begin{cases}
    \partial_t \mu_t + \operatorname{div}_x \{\mu_t \nabla_p H(-\nabla v(t, x))\} = 0,\\
    \mu_0 \sim \delta_{\int x m_0(x) dx}.
\end{cases}
\end{equation}
Assume further that $\nabla_p H(-\nabla v(t,x))$ is Lipschitz in $x$ bounded in $t$ \footnote{ 
The bound \eqref{eq:bdmu} is a conditional result on the Lipschitz assumption \eqref{eq:Lipkey}.
The assumption implicitly requires $v$ to be a classical solution, 
which is not true in general.
A sufficient condition for this assumption to hold is that $H$ is uniformly convex, $F, G$ are convex, 
and $H, F, G$ are smooth with bounded Hessians,
i.e., $c_H I \le \nabla^2_{pp}H \le C_H I$, $0 \le \nabla^2_{xx}F \le C_F$ and $0 \le \nabla^2_{xx} G \le C_G$.
The uniform convexity of $H$, the convexity of $F, G$ and their smoothness imply that $v$ is a classical solution because the characteristics do not cross. 
The boundedness of $\nabla^2_{xx} F$ and $\nabla^2_{xx}G$ further guarantees that $0 \le \nabla^2_{xx} v \le \max\left(C_G, \sqrt{C_F/c_H}\right) I$.
Combined with the boundedness of $\nabla^2_{pp}H$ yields the assumption.}, i.e., there is $L > 0$ such that
\begin{equation}\label{eq:Lipkey}
|\nabla_p H(-\nabla v(t,x)) - \nabla_p H(-\nabla v(t,y))| \le L |x-y| \quad \mbox{for all } t,x,y.
\end{equation}
Then there exists a constant $C > 0$ (independent of $N$) such that
\begin{equation} \label{eq:bdmu}
    W_2(\mu^N_t, \mu_t) \le C\left( \frac{1}{\sqrt{N}} + \frac{\sqrt{\log N}}{N} + \sqrt{\frac{\mbox{Tr} \, \Sigma}{N}}\right).
\end{equation}
\end{corollary}
\begin{proof}
First observe that the pair $(v^N, \mu^N)$ solves the (degenerate) MFG:
\begin{equation}
\begin{cases}
    -\partial_t v^N(t, x) + H(-\nabla v^N(t, x)) - F(x) = \frac{\beta^2}{N} \Delta v^N(t, x), \\
    \partial_t \mu^N_t + \operatorname{div}_x \{\mu^N_t \nabla_p H(-\nabla v^N(t, x))\} = \frac{\beta^2}{N} \Delta \mu^N_t,\\
    v^N(T, x) = G(x), \quad \mu^N_0 = \operatorname{Law}(\overline{X}^N_0),
\end{cases}
\end{equation}
Note that the HJ equation is not coupled with $\mu^N$. Let $(\tilde{v}^N, \tilde{\mu}^N)$ be a solution to the MFG:
\begin{equation*}
\begin{cases}
    -\partial_t \tilde{v}^N(t, x) + H(-\nabla \tilde{v}^N(t, x)) - F(x) = 0, \\
    \partial_t \tilde{\mu}^N_t + \operatorname{div}_x \{\tilde{\mu}^N_t \nabla_p H(-\nabla \tilde{v}^N(t, x))\} = 0,\\
    \tilde{v}^N(T, x) = G(x), \quad \tilde{\mu}^N_0 = \operatorname{Law}(\overline{X}^N_0).
\end{cases}
\end{equation*}
(So $\tilde{v}^N = v$ in the equation \eqref{eq:vHJ}.)
As a consequence of Corollary \ref{convergence_of_rho}, we obtain:
\begin{equation} \label{eq:mu1}
    W_2(\mu^N_t, \tilde{\mu}^N_t) \le \frac{C}{\sqrt{N}} \quad \mbox{for some } C > 0 \mbox{ (independent of } N).
\end{equation}
By the Lipschitz assumption \eqref{eq:Lipkey} 
and the Cauchy-Lipschitz theory of the continuity equation 
(see e.g., \cite[Section 2]{AC14}):
\begin{equation} \label{eq:mu2}
    W_2(\tilde{\mu}^N_t, \mu_t) \le C W_2\left(\operatorname{Law}(\overline{X}^N_0), \delta_{\int x m_0(x) dx}\right).
\end{equation}
Without loss of generality, assume that $X_0^i$ has mean $0$, i.e., $\int x m_0(x) dx  = 0$. We have:
\begin{equation} \label{eq:mu3}
    \begin{aligned}
    W_2\left(\operatorname{Law}(\overline{X}^N_0), \delta_0 \right)
    & \le W_2\left(\operatorname{Law}(\overline{X}^N_0), \mathcal{N}\left(0, \frac{\Sigma}{N}\right)\right) + W_2\left(\mathcal{N}\left(0, \frac{\Sigma}{N}\right),\delta_{0}\right) \\
    & \le \frac{C \sqrt{d \log N}}{N} + \sqrt{\frac{\mbox{Tr} \Sigma}{N}},
    \end{aligned}
\end{equation}
where the first term in the last inequality follows from \cite[Theorem 1]{EMZ20} \footnote{A slightly looser bound $\mathcal{O}(\sqrt{d} \log N/N)$ (up to a $\log N$ factor) was proved in \cite[Theorem 1.1]{Zhai18}.},
and the $W_2$ distance of two Gaussian vectors (see e.g., \cite[Proposition 7]{GS1984}). 
Combining the equations \eqref{eq:mu1}, \eqref{eq:mu2} and \eqref{eq:mu3} yields
the desired bound.
\end{proof}


\subsection{Policy iteration}

As mentioned in the Introduction, there has been growing interest in first-order MFG models, but solving first-order MFGs numerically poses challenges.

\quad Policy iteration (PI) is a class of approximate dynamic programming algorithms
that have been used to solve stochastic control problems
with provable guarantees \cite{GTZ25, KSS20, MWZ24, TWZ25, TZ24}.
In a series of papers \cite{CCG21, Cam22, CT22}, PI was proposed to solve
second-order MFGs with separable Hamiltonians.
An extension to second-order MFGs with non-separable Hamiltonians
was considered in \cite{Lauriere2022}.
However, PI is not directly applicable to the first-order problems
due to ill-posedness \cite{TTZ25}. 
So a reasonable idea is to approximate first-order MFGs by second-order
MFGs \footnote{This idea was also proposed in \cite{TTZ25} to solve deterministic control problems by PI.}, 
and a convergence rate of second-order MFGs to the vanishing viscosity limit
gives the approximation error.

\quad Now, let us specify the PI for solving the MFG \eqref{eq:mfg} with $\beta > 0$. For simplicity, we assume that the terminal data $g(x,\rho) = g(x)$ depend only on $x$. There are three steps: given $R > 0$ and a measurable function $q^0: [0,T] \times \mathbb{R}^d \to \mathbb{R}^d$ with $||q^0||_\infty \le R$, we iterate for $n \ge 0$,
\begin{itemize}[itemsep = 3 pt]
\item[(i)]
Solve 
\begin{equation} \label{eq:PI1}
    \partial_t \rho_t^{n, \beta} -\operatorname{div}\{\rho_t^{n, \beta} q^n \} = \frac{\beta^2}{2} \Delta \rho_t^{n, \beta}, \quad \rho_0^{n, \beta} = m_0.
\end{equation}
\item[(ii)]
Solve 
\begin{equation} \label{eq:PI2}
    -\partial u^{n, \beta} + q^n \nabla u^{n, \beta} -\mathcal{L}(x,-\nabla u^{n, \beta}, q^n, \rho_t^{n, \beta}) = \frac{\beta^2}{2} \Delta u^{n, \beta}, \quad u^{n, \beta}(T,x) = g(x),
\end{equation}
where $\mathcal{L}(x, p,q, \rho):= p\cdot q - H(x,p,\rho)$.
\item[(iii)]
Update the policy
\begin{equation} \label{eq:PI3}
q^{n+1}(t,x):= {\arg \max}_{|q| \le R} \left(q \cdot \nabla u^{n, \beta}(t,x) - \mathcal{L}(x, q, \rho_t^{n, \beta})\right),
\end{equation}
where $\mathcal{L}(x,q, \rho):=\max_p \mathcal{L}(x,p,q,\rho)$.
\end{itemize}

\quad In all of the aforementioned works \cite{CCG21, Cam22, CT22, Lauriere2022},
the convergence (rate) of PI \eqref{eq:PI1}--\eqref{eq:PI3} for MFGs 
was proved on the torus $\mathbb{R}^d/\mathbb{Z}^d$,
rather than the whole space $\mathbb{R}^d$ to avoid boundary effects. 
Nevertheless, a review of the methods in these papers 
allow to prove the convergence of PI for solving MFGs on $\mathbb{R}^d$.
The extension is technical, and goes beyond the scope of this paper.
The claim below, extending \cite{Lauriere2022}, summarizes the ``expected" convergence results of PI for solving second-order MFGs on $\mathbb{R}^d$.
We plan to prove it rigorously in the future.
\begin{claim} \label{cl:PI}
Under suitable conditions on $H(x,p,\rho)$, $m_0(x)$ and $g(x)$
 (e.g., $H$ and its derivatives are Lipschitz and $H$ is strictly convex in $p$,
 and $m_0$, $g$ have some Sobolev regularity),
 for any compact set $\mathcal{K} \subset \mathbb{R}^d$, 
 there exists $T = T(\mathcal{K}, \beta) > 0$ and $C = C(\mathcal{K}, \beta)$ such that 
 \begin{equation}
     ||u^{n, \beta} - u^\beta||_{W^{1,2}_r([0,T] \times \mathcal{K})} + 
     ||\rho^{n, \beta} - \rho^\beta||_{W^{1,2}_r([0,T] \times \mathcal{K})}
     \le C e^{-n}, \quad \mbox{for } r > d+2,
 \end{equation}
where $W_r^{1,2}(Q)$ denotes the space of functions $f$ such that $\partial_t^\delta \partial_x^{\sigma}f \in L^r(Q)$ for all multi-indices $(\delta, \delta')$ with $2 \delta + \delta' \le 2$,
and
\begin{equation*}
    ||f||_{W_r^{1,2}(Q)}:= \left(\int_Q \sum_{2 \delta + \delta' \le 2} |\partial_t^\delta \partial_x^{\delta'}f|^r dt dx\right)^{\frac{1}{r}}.
\end{equation*}
The constants $T(\mathcal{K}, \beta), C(\mathcal{K}, \beta) > 0$ depend on $\mathcal{K}, \beta$ in a complicated way. 
Given $\mathcal{K}$ and as $\beta \to 0$, $C(\mathcal{K}, \beta)$ is typically of order $e^{\frac{C}{\beta^2}}$ for some $C>0$, 
and $T(\mathcal{K}, \beta)$ is typically of order $\beta^{-\kappa}$ for some $\kappa > 0$.
\end{claim}

\quad With Claim \ref{cl:PI} in place, 
we derive the (time-weighted) convergence rate of $u^{\beta,n}$ to $u$ by simply applying the triangle inequality.
\begin{corollary} \label{coro:PI}
    Let $\mathcal{K} \subseteq \mathbb{R}^d$ be a compact set. 
    Under the assumptions in Theorem \ref{main_result} and Claim \ref{cl:PI}, there exist $T = T(\beta) > 0$, $C_1 = C_1(\mathcal{K})$ and $C_2 = C_2(\beta)$ such that
    \begin{equation}
    \frac{1}{T}||u^{n, \beta} - u||_{L^r([0,T] \times \mathcal{K})} \le C_1 \beta + C_2(\beta) e^{-n} \quad \mbox{for } r > d+2.
    \end{equation}
\end{corollary}

\quad As a consequence of Corollary \ref{coro:PI}, we get the complexity of PI for solving the first-order MFGs. 
Assume that an accuracy of $\varepsilon > 0$ is required,
i.e., $\frac{1}{T}||u^{n, \beta} - u||_{L^r([0,T]\times \mathcal{K})}  \le \varepsilon$.
Then we set:
\begin{equation}
    \beta \asymp \varepsilon \quad \mbox{and} \quad C_2(\beta) e^{-n} \asymp \varepsilon,
\end{equation}
so $n \asymp \log(C(\varepsilon)/\varepsilon)$.
The discussion at the end of Claim \ref{cl:PI} suggests that
$C_2(\varepsilon)$ be of order $e^{\frac{C}{\varepsilon^2}}$ for some $C > 0$, 
as $\varepsilon \to 0$.
Therefore, we have $n \asymp \varepsilon^{-2}$, 
i.e., it takes the order of $\varepsilon^{-2}$ steps for PI to approximate
$u^0$ with accuracy $\varepsilon$.


\section{Examples and numerical results} \label{sc6}

\subsection{A closed-form example} \label{sc61}

As mentioned in the introduction, the convergence rate of vanishing viscosity approximations to MFGs matches the classically optimal rate of that to HJ equations, so it is hard to expect a better rate in the general setting. Nevertheless, this does not rule out some MFGs with special structures, which may have sharper rates of convergence.

\quad Consider the following example from \cite{CCS24, CSZ24}:
\begin{equation}\label{eq:closedmfg}
\begin{cases}
    -\partial_t u^{\beta} + \frac{1}{2} |\nabla u^\beta|^2 - \frac{1}{2} \left(x- \int y \rho^\beta_t(y)dy \right)^2 = \frac{\beta^2}{2} \Delta u  & \text{ on } [0, T] \times \mathbb{R}^d, \\
    \partial_t \rho^{\beta}_t - \operatorname{div}_x\{\rho^{\beta}_t \nabla u^\beta\} = \frac{\beta^2}{2}\Delta \rho^{\beta}_t & \text{ on } [0, T] \times \mathbb{R}^d, \\
    u^{\beta}(T, x) = 0, \quad \rho^{\beta}_0(x)\sim \mathcal{N}(m, \sigma^2 I) & \text{ on } \mathbb{R}^d.
\end{cases}
\end{equation}
That is, the MFG \eqref{eq:closedmfg} has a nonlocal and separable Hamiltonian
\begin{equation}
    H(x,p,\mu) = \frac{1}{2} |p|^2 - \frac{1}{2} \left(x - \int_y y \mu(y) dy\right)^2,
\end{equation}
with $g(x, \mu) = 0$ and $m_0(x)$ being Gaussian with mean $m$ and covariance matrix $\sigma^2 I$. Interestingly, this MFG has a closed-form solution:
\begin{equation}
    u^\beta(t,x) = \frac{e^{2T-t} - e^t}{2(e^{2T-t} + e^t)}|x -m|^2 - \frac{\beta^2 d}{2} \ln\left( \frac{2 e^T}{e^{2T-t} + e^t}\right),
\end{equation}
and
\begin{equation}
    \rho^\beta_t(x) \sim \mathcal{N}\left(m, \left(\sigma^2 \left(\frac{e^{2T-t} + e^t}{e^{2T} + 1} \right)^2 + \beta^2 \frac{(e^{2T-t}+e^t)^2 (e^{2t}-1)}{2(e^{2T} + 1)(e^{2T} + e^{2t})}\right) I \right).
\end{equation}
As a consequence,
\begin{equation}
    ||u^\beta - u||_\infty \le C \beta^2 \quad \mbox{and} \quad 
    W_1(\rho_t^\beta, \rho_t) \le C \beta^2,
\end{equation}
for some $C > 0$ (independent of $\beta$). 
The same rate $\mathcal{O}(\beta^2)$ for vanishing viscosity 
may also be extended to a class of displacement monotone MFGs 
by using the arguments in \cite{CM24}.

\subsection{Numerical examples} \label{sc62}
We proved in Theorem \ref{main_result} that $u^\beta$ of MFGs with a nonlocal Hamiltonian converges at a rate of $\mathcal{O}(\beta)$. Here we compare the rate to that of MFGs with a local coupling. 

\quad We consider the following example on $[0, 0.25] \times \mathbb{T}^1$ (i.e., $T = 0.25$) with:
\begin{equation}
    H(x, p, \mu(x)) = 0.01 \left\{|p|^2 - \mu(x)^2 - \cos(4 \pi x) - 0.1 \cos(2 \pi x) - 0.1 \sin\left(2 \pi \left(x - \frac{\pi}{8}\right)^2 \right)\right\},
\end{equation}
and $g(x) = 0$, and $m_0$ being Gaussian center at $0$ with variance $0.01$ truncated to have Dirichlet boundary conditions. Figure \ref{fig:1} plots the solutions to this local and separable MFG, with $\beta \in \{0.1, 0.3, 0.5, 1.0\}$, and Figure \ref{fig:2} illustrates how $||u^\beta - u||_\infty$ varies against $\beta$ (for $\beta \in \{0.1, 0.2, \ldots, 0.9, 1\}$). To solve the MFG, we used Picard iteration and added damping for stabilization purposes, with every iteration first solving for the Fokker-Planck equation and then the HJB equation. Since the Fokker-Planck equation is linear, we can use a generic linear solver for the system of equations derived from the equation's finite difference representation, but since the HJB equation is nonlinear, we used Newton's method to solve its system of equations \footnote{Our numerical results are based on the codes available at \url{https://colab.research.google.com/drive/1shJWSD2MA5Fo7_rB625dAvNTdZS1a7bG?usp=sharing}.}.

\begin{figure}[h]
\centering
\begin{subfigure}{0.24\textwidth}
  \centering
  \includegraphics[width=\linewidth]{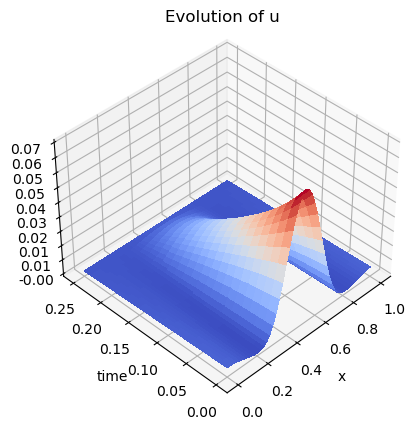}
\end{subfigure}%
\begin{subfigure}{0.24\textwidth}
  \centering
  \includegraphics[width=\linewidth]{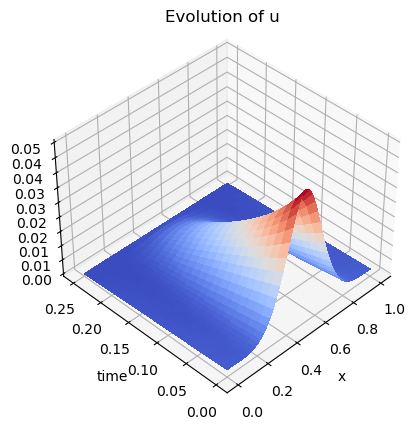}
\end{subfigure}
\begin{subfigure}{0.24\textwidth}
  \centering
  \includegraphics[width=\linewidth]{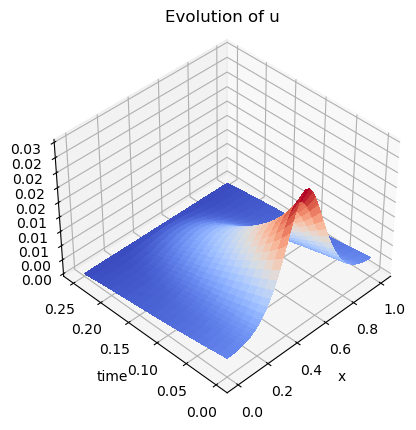}
\end{subfigure}
\begin{subfigure}{0.24\textwidth}
  \centering
  \includegraphics[width=\linewidth]{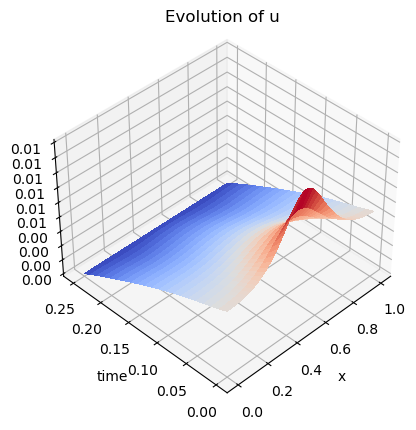}
\end{subfigure}
\begin{subfigure}{0.24\textwidth}
  \centering
  \includegraphics[width=\linewidth]{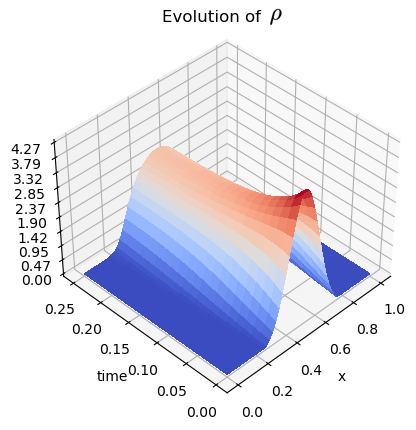}
\end{subfigure}%
\begin{subfigure}{0.24\textwidth}
  \centering
  \includegraphics[width=\linewidth]{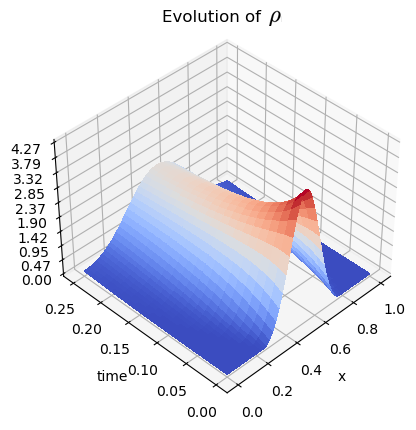}
\end{subfigure}
\begin{subfigure}{0.24\textwidth}
  \centering
  \includegraphics[width=\linewidth]{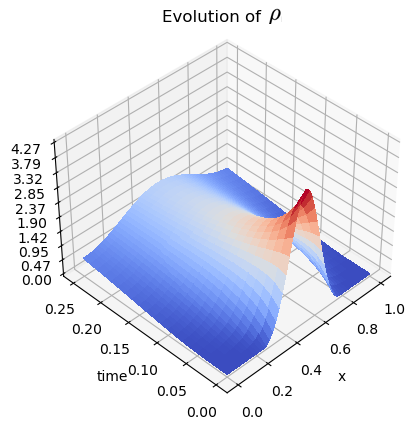}
\end{subfigure}
\begin{subfigure}{0.24\textwidth}
  \centering
  \includegraphics[width=\linewidth]{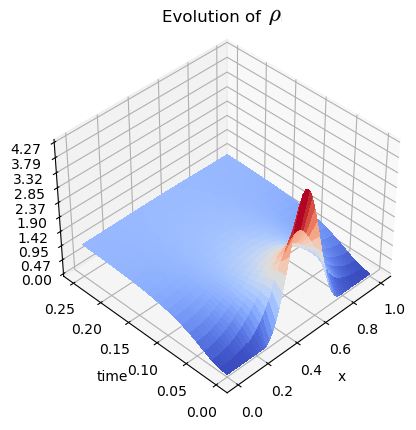}
\end{subfigure}
\caption{Plot of $(u^\beta, \rho^\beta)$ for $\beta \in \{0.1, 0.3, 0.5, 1.0\}$ (left to right).}
\label{fig:1}
\end{figure}

\begin{figure}[h]
\centering
\includegraphics[width=0.4\linewidth]{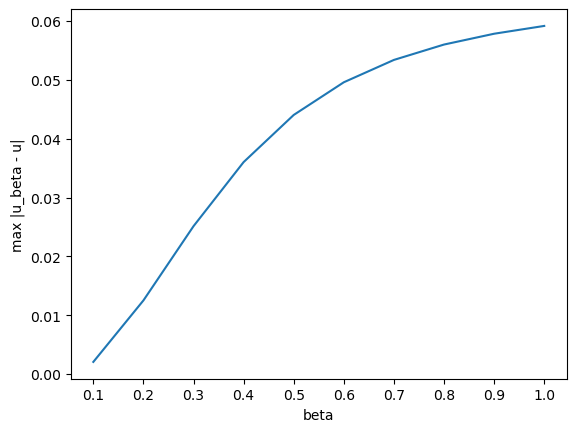}
\caption{Plot of $||u^\beta - u||_\infty$ again $\beta$.}
\label{fig:2}
\end{figure}

\quad In \cite{Tang2023b}, it was proved that $u^\beta$ converges at a rate of $\mathcal{O}(\beta^{\frac{1}{4}})$ in some weighted $L^2$ norm. Now by regressing $\log ||u^\beta - u||_\infty$ over $\log \beta$, we find that the slope is $1.050$ using all $\beta \in \{0.1, \ldots, 1.0\}$, while the slope is $1.162$ using the first half $\beta \in \{0.1, \ldots, 0.5\}$. It is natural to expect that
\begin{equation} \label{eq:expfit}
    ||u^\beta - u||_\infty \asymp \beta^{1+\delta} \quad \mbox{as } \beta \to 0,
\end{equation}
for some $0 < \delta < 1$. The rate \eqref{eq:expfit} is better than the proved $\mathcal{O}(\beta^{\frac{1}{4}})$-rate for MFGs with a local Hamiltonian, and is between the $\mathcal{O}(\beta)$-rate for MFGs with a general nonlocal Hamiltonian and the $\mathcal{O}(\beta^2)$-rate for the example in Section \ref{sc61}. An interesting question is to find suitable conditions on model data to achieve the rate in \eqref{eq:expfit} (with an explicit $\delta$), hence improving the bounds in \cite{Tang2023b}.

\section{Conclusion} \label{sc7}

\quad This paper studies the convergence rate of the vanishing viscosity approximation to MFGs with a nonlocal, and possibly non-separable Hamiltonian. With $\beta^2$ as the diffusivity constant, we prove that $u^\beta$ and $\rho^\beta$ converge a rate of $\mathcal{O}(\beta)$ in the topology of uniform convergence on compact sets and the $W_2$ metric, respectively. Our approach exploits both probabilistic and analytical arguments, where the FBSDE representation of the MFG is used to derive the convergence rate of $\rho^\beta$, and the rate of $u^\beta$ follows from a stability property of the HJB equation. We also apply our result to $N$-player games, mean field control, and policy iteration for MFGs.

\quad There are several directions to extend this work:
\begin{enumerate}[itemsep = 3 pt]
    \item
    First, our result is proved for MFGs with a nonlocal and possibly non-separable Hamiltonian. It would be interesting to establish the convergence result for MFGs with a local Hamiltonian, underpinning the numerical results in Section \ref{sc62}.
    \item
    Second, we prove in this work the convergence rate of vanishing viscosity for MFGs in $\mathbb{R}^d$; while \cite{TT23} considered the case in $\mathbb{T}^d$. The main difference between these two papers is that our work uses an FBSDE representation of the MFG together with a PDE stability result, while \cite{TT23} relies exclusively on PDE arguments. A natural question is whether the FBSDE approach can be extended to other domains, so that the convergence can be established for MFGs on domains other than $\mathbb{T}^d$ and $\mathbb{R}^d$.
    \item
    Finally, the vanishing viscosity approximation to MFGs can be regarded as a ``perturbation" of first order MFGs, where the perturbation is to add the operator $\frac{\beta^2}{2} \Delta$. We expect that the tools in this paper can also be used to analyze other types of perturbation, e.g., perturbation on the Hamiltonian. A notable example is the entropy-regularized relaxed control \cite{WZZ20} in the context of reinforcement learning, where the HJB equation is replaced with the exploratory equation under entropy regularization \cite{TZZ22}.
\end{enumerate}



\bigskip
{\bf Acknowledgment:} 
We thank the referee for their careful reading and valuable suggestions, which have greatly improved this work. We thank Alp\'{a}r M\'{e}sz\'{a}ros for pointing out an error in a lemma in a previous version of this work concerning the stability of McKean-Vlasov FBSDEs, and we also thank Daniel Lacker for helpful discussions on how to resolve the issue.
W. Yu is supported by the Columbia Innovation Hub grant and NSF grant DMS-2309245.
Q. Du is supported in part by NSF grants DMS-2309245 and DMS-1937254.
W. Tang is supported by NSF CAREER Award DMS-2538791, the Tang Family Assistant Professorship and a Columbia-CityU/HK collaborative project that is supported by InnoHK Initiative, The Government of the HKSAR and the AIFT Lab.
\bibliographystyle{abbrv}
\bibliography{references} 

\appendix

\end{document}